\documentclass[11pt,reqno]{amsart}
\usepackage{amssymb,mathtools,enumitem,mathrsfs,cite,fullpage}
\usepackage{booktabs,tabularx,float}
\usepackage{tikz}
\usepackage{amsthm,amsmath}
\usepackage[T1]{fontenc}
\usepackage{lmodern}
\usepackage{hyperref,refcount,microtype}

\newtheorem{theorem}{Theorem}[section]
\newtheorem{lemma}[theorem]{Lemma}
\newtheorem{corollary}[theorem]{Corollary}
\newtheorem{proposition}[theorem]{Proposition}
\newtheorem{claim}{Claim}
\newcommand{\claimparent}{0}
\renewcommand{\theclaim}{\claimparent.\arabic{claim}}

\newcommand{\claimsof}[1]{\setcounter{claim}{0}\edef\claimparent{\getrefnumber{#1}}}
\newenvironment{poc}{\begin{proof}[Proof of Claim~\theclaim]}{\end{proof}}
\theoremstyle{definition}
\newtheorem{definition}[theorem]{Definition}
\theoremstyle{remark}

\newcommand{\N}{\mathbb N}
\newcommand{\Z}{\mathbb Z}
\newcommand{\Q}{\mathbb Q}
\newcommand{\R}{\mathbb R}
\newcommand{\E}{\mathbb E}
\renewcommand{\Pr}{\mathbb P}
\newcommand{\Var}{\operatorname{Var}}

\renewcommand{\leq}{\leqslant}

\renewcommand{\le}{\leqslant}
\renewcommand{\ge}{\geqslant}
\setlist[enumerate,1]{label=\textup{(\roman*)},leftmargin=*}

\newcommand{\rank}{\operatorname{rank}}
\newcommand{\Span}{\operatorname{span}}

\newcommand{\canarrow}{\xrightarrow{\mathrm{can}}}
\newcommand{\cI}{\mathcal I}

\title[Extensions and canonical Rado thresholds]{Extending functions on random sets\\ and canonical Rado thresholds}
\author{Zhenzhen Lou}
\address{College of Science, University of Shanghai for Science and
Technology, Shanghai 200093, China}
\email{louzz@usst.edu.cn}

\author{Guangfu Wang}
\address{School of Mathematics and Information Sciences, Yantai University, Shandong 264005, China}
\email{gfwang@ytu.edu.cn}

\author{Man Wu}
\address{School of Mathematics and Computational Science, Xiangtan University, Hunan 411105, China}
\email{manwu@smail.xtu.edu.cn}

\author{Zhifei Yan}
\address{ECOPRO, Institute for Basic Science, 55 Expo-ro, Yuseong-gu,
Daejeon, 34126, Republic of Korea}
\email{zhifeiyan@ibs.re.kr}

\date{}

\hypersetup{colorlinks=false,pdfborder={0 0 1},
linkbordercolor={1 0 0},citebordercolor={0 1 0},
pdftitle={Extending functions on random sets and canonical Rado thresholds},
pdfauthor={Zhenzhen Lou, Guangfu Wang, Man Wu, Zhifei Yan}}

\begin{document}

\begin{abstract}
Rado characterised the systems of linear equations for which every
finite colouring of the positive integers contains a monochromatic
solution. Lefmann extended this to unrestricted colourings, with
canonical solutions that are monochromatic, rainbow, or have a
prescribed block pattern. For the special case of arithmetic
progressions, Alvarado, Kohayakawa, Morris, Mota and Ortega recently
proved a sparse canonical van der Waerden theorem on random host $[n]_p$.

In this paper, we prove a sparse canonical Rado theorem on $[n]_p$ for every
irredundant partition-regular system of positive rank and every
prescribed columns-condition partition. The threshold scale
$n^{-1/m(A)}$ agrees with that of the fixed-colour random Rado
theorem established by R\"odl and Ruci\'nski, and Friedgut, R\"odl and Schacht. A matching
$0$-statement holds except for Schur's equation, for which we
determine the limiting canonical probability throughout the
critical regime.

Interestingly, for the lower bound, we prove an extension theorem for regular
linear systems over arbitrary fields, which might be of independent interest. With high probability,
every subset of a sufficiently sparse random sample contains a
positive proportion of points on which every prescribed function
extends to the whole sample while preserving all proper solution
relations. Applied to additive relations, this also gives the
asymptotic Freiman dimension of sparse random sets in a
small-constant regime, making partial progress on a question of
Conlon and Gowers.
\end{abstract}
\maketitle

\section{Introduction}

Arithmetic Ramsey theory asks which configurations must occur
in a single colour class of every finite colouring of the
positive integers. Van der Waerden's theorem~\cite{VDW} guarantees
monochromatic arithmetic progressions of every length, while Rado~\cite{Rado}
characterised the systems $A\mathbf x=0$ with this property.
The corresponding integer matrices are called
\emph{partition regular}.

With unrestricted colours, giving each integer a different colour
rules out monochromatic solutions with distinct coordinates,
but makes every such solution \emph{rainbow}. Erd\H{o}s and
Graham~\cite{EG} proved that these two alternatives suffice for
progressions. For general partition-regular systems, Lefmann's canonical
theorem~\cite{Lefmann} also allows a prescribed block pattern:
coordinates in the same block have the same colour, and
different blocks have different colours.

How sparse can a set be and still force a canonical solution?
Let $[n]_p$ retain each element of $[n]=\{1,\ldots,n\}$
independently with probability $p$. Alvarado, Kohayakawa, Morris,
Mota and Ortega~\cite{AKMMO} proved that, for $k$-term
progressions, $k\ge3$, the canonical property has threshold
scale $n^{-1/(k-1)}$, just as for a fixed number of colours.
They asked for a sparse version of Lefmann's theorem for
general systems.

We answer this question. For every irredundant partition-regular
system of positive rank, the threshold scale for canonical
solutions agrees with that for monochromatic solutions in a
fixed number of colours. A matching
small-constant $0$-statement holds except for Schur's equation,
whose critical probability curve we determine. Our lower bound
rests on a general extension theorem for functions on sparse
random sets over arbitrary fields.

The prescribed blocks come from Rado's algebraic criterion.
A solution is \emph{proper} if its coordinates are pairwise
distinct, and $A$ is \emph{irredundant} if it has a proper
positive integer solution. Remove redundant rows and write
$A=(\mathbf a_1,\ldots,\mathbf a_k)\in\Z^{\ell\times k}$
with $\rank A=\ell\ge1$. The \emph{columns condition} asks for
an ordered partition $\cI=(I_0,\ldots,I_s)$ of $[k]$ into
nonempty blocks such that
\begin{equation}\label{eq:columns}
 \sum_{i\in I_0}\mathbf a_i=0,\qquad
 \sum_{i\in I_t}\mathbf a_i\in
 \Span_{\Q}\{\mathbf a_i:i\in I_0\cup\cdots\cup I_{t-1}\}
 \quad(1\le t\le s).
\end{equation}
Rado's theorem says that this is equivalent to partition
regularity. Ranks and spans are over $\Q$ unless another field
is specified. Fix $\cI$ before choosing the colouring. A proper
solution is \emph{canonical} if it is monochromatic, rainbow,
or constant on each $I_t$ with different colours on different
blocks. For Schur's equation $x+y=z$, the partition
$(\{x,z\},\{y\})$ gives the additional pattern
$\chi(x)=\chi(z)\ne\chi(y)$.
Write $U\canarrow A$ if every colouring of $U\subseteq\N$
contains a proper canonical solution. Lefmann's theorem,
using Deuber's configurations~\cite{Deuber}, gives
$\N\canarrow A$ for every irredundant partition-regular system
and every prescribed columns-condition partition.

The reference scale is given by the fixed-colour random Rado
theorem. R\"odl and Ruci\'nski~\cite{RRvdW} determined the
fixed-colour progression threshold and introduced the density
parameter for general systems~\cite{RR}. If $A_{\overline J}$
denotes the submatrix on columns outside $J\subseteq[k]$, set
\begin{equation}\label{eq:m(A)}
 m(A)=\max_{\substack{J\subseteq[k]\\|J|\ge2}}
 \frac{|J|-1}{|J|-1+\rank(A_{\overline J})-\ell}.
\end{equation}
For irredundant partition-regular $A$ with $k\ge3$, all
denominators are positive and $m(A)>1$. Progressions have
$m(A)=k-1$, while Schur's equation has $m(A)=2$.
For each fixed number $r\ge2$ of colours, R\"odl and
Ruci\'nski~\cite{RR} proved the $0$-statement at
$n^{-1/m(A)}$, and Friedgut, R\"odl and Schacht~\cite{FRS}
completed the general $1$-statement, with constants depending
on $r$. For the $0$-statement we use the corrected proof of
Hancock and Treglown~\cite{HT}.

For the $1$-statement, the fixed-colour argument of Friedgut,
R\"odl and Schacht~\cite{FRS} uses induction on the number of
colours, with constants depending on that number. It therefore
does not directly give a bound for arbitrary palettes.
We keep large colour classes unchanged and merge small ones
into a bounded palette, then seek canonical solutions whose
repeated colours all come from unchanged classes. A dense count
whose constants do not depend on the palette then allows us to
apply containers and force such a solution in the sample. Its
colour pattern survives restoration of the original colours.

For the $0$-statement, a two-colouring without monochromatic
solutions also avoids rainbow solutions, but may realise the
prescribed two-block pattern. Splitting classes can remove that
pattern while creating rainbow solutions, so choices must be
coordinated across overlapping solutions. This leads to a
structural result: an extension theorem describing the freedom
left by the solution relations, uniformly inside every subset
of a sparse random sample. We present this theorem first.

\subsection{Extending functions on random sets}

The extension question asks how many values of a function can
be chosen freely while preserving every proper solution of a
linear system. To build our colouring, we need enough freedom
inside every subset of the sample, so that we can partition the
sample into a bounded number of sets on which values can be
prescribed arbitrarily. A bound for the whole sample does not
suffice: solutions using points outside a subset can impose
relations on it, even when it contains no complete solution.

The following theorem gives this freedom for regular systems
over arbitrary fields; partition regularity is not required.

For the statement, write $B[H]$ for the submatrix of $B$ on
columns $H$. We call a full-row-rank matrix
$B\in\mathbb F^{r\times s}$, with $r\ge1$ and $s\ge3$,
\emph{regular} if
\[
 (s-1)\rank B[H]>r|H|
 \qquad\text{for every $H\subseteq[s]$ with $1\le |H|\le s-2$}.
\]
Equivalently, $B$ is strictly balanced and deleting any two
columns preserves its rank; see Section~\ref{sec:sparse-rank}.
For a finite set $S$, let $S_p$ retain its elements independently
with probability $p$.

\begin{theorem}[Extension theorem]\label{thm:rank-informal}
Let $B$ be a regular matrix of rank $r\ge1$ with $s\ge3$ columns
over a field $\mathbb F$, and set $m=(s-1)/(s-r-1)$.
There is $c_s>0$ such that, for every $N$-element set
$S\subseteq\mathbb F\setminus\{0\}$ and $0\le p\le c_sN^{-1/m}$,
with high probability every $X\subseteq S_p$ contains a set
$Y\subseteq X$ with 
$$|Y|\ge |X|/m$$
such that every map
$g:Y\to\mathbb F$ extends to a map $f:S_p\to\mathbb F$ satisfying
\[
 B(x_1,\ldots,x_s)^{\mathsf T}=0
 \quad\Longrightarrow\quad
 B(f(x_1),\ldots,f(x_s))^{\mathsf T}=0
\]
for all pairwise distinct $x_1,\ldots,x_s\in S_p$.
\end{theorem}

The set $Y$ is chosen before its values are prescribed, and the
values $f(x_i)$ need not be distinct. For Schur's equation,
$m=2$, so at least half of every subset admits arbitrary
prescriptions.

\begin{corollary}[Schur extension]\label{cor:schur-extension}
There is $c>0$ such that, if $0\le p\le cn^{-1/2}$ and
$\Omega=[n]_p$, then with high probability every $X\subseteq\Omega$
contains a set $Y\subseteq X$ with $|Y|\ge |X|/2$ such that
every map $g:Y\to\R$ extends to a map $f:\Omega\to\R$ satisfying
\[
 f(x)+f(y)=f(z)
 \qquad\text{whenever $x,y,z\in\Omega$ are distinct and $x+y=z$}.
\]
\end{corollary}

For additive relations, take $B=(1,1,-1,-1)$, so $m=3/2$.
The resulting extension property also controls the number of
independent values of an additive map. It yields the
Freiman-dimension formula below in a small-constant regime,
making partial progress on a question of Conlon and
Gowers~\cite{CGFreiman}.

\begin{corollary}[Freiman extension]\label{cor:additive-extension}
There is $c_0>0$ such that, if $0\le p\le c_0n^{-2/3}$ and
$\Omega=[n]_p$, then with high probability every $X\subseteq\Omega$
contains a set $Y\subseteq X$ with $|Y|\ge2|X|/3$ such that
every map $g:Y\to\R$ extends to a map $f:\Omega\to\R$ satisfying
\[
 f(a)+f(b)=f(u)+f(v)\qquad
 \text{whenever $a,b,u,v\in\Omega$ are distinct and $a+b=u+v$.}
\]
Moreover, if $p=cn^{-2/3}$ for a fixed $0<c\le c_0$, then the
Freiman dimension of $\Omega$ satisfies
\[
 d_{\rm Fr}(\Omega)=(1-c^3/12+o_{\Pr}(1))|\Omega|.
\]
\end{corollary}

The definition of Freiman dimension and the additional argument
for the dimension formula are given in Appendix~\ref{app:freiman},
together with the cyclic-group version. We now return to the
colouring problem that motivated extension.

\subsection{Canonical Rado thresholds}

For a non-Schur system with a prescribed two-block pattern,
we project its solutions onto a regular density core $B$ with
$m(B)=m(A)$. Every proper $A$-solution contains a proper core
solution, and the induced partition has one or two blocks.
Thus we can use the extension theorem at the original random
scale. Its bound for every subset allows us to apply Edmonds'
vector partition theorem~\cite{Edmonds}, partitioning the sample
into sets that admit arbitrary prescriptions. No such part can
contain a proper core solution, since that would impose a
nonzero equation on freely chosen values. We use the restrictions
on how solutions meet these parts to choose colours avoiding
all three canonical patterns. Three-variable cores also require
an analysis of overlapping solutions. Together with the ordinary
two-colour argument for the other partitions, this proves the
non-Schur direction of the following characterization.

\begin{theorem}[0-statement]\label{thm:zero}
Let $A\in\Z^{\ell\times k}$ be an irredundant partition-regular
matrix with $k\ge3$ and $\rank(A)=\ell\ge1$, and fix a partition
$\cI$ witnessing the columns condition. The following are equivalent:
\begin{enumerate}
\item $A$ is not a non-zero scalar and column-permutation copy of
$(1,1,-1)$.
\item There is $c>0$ such that, for every sequence
$p=p(n)\in[0,1]$ with $p\le cn^{-1/m(A)}$,
\[
 \Pr\bigl([n]_p\canarrow A\bigr)=o(1).
\]
\end{enumerate}
\end{theorem}

Schur's equation explains why the exception is necessary.
It has the extension property of Corollary~\ref{cor:schur-extension},
but its colour requirements can conflict on just four points.
A canonical-free colouring of a proper triple $a+b=z$ must give
$a,b$ the same colour and $z$ a different one. The proper triple
$a+z=2a+b$ requires $a,z$ to have the same colour. Thus, for
distinct positive $a,b$, the set $\{a,b,a+b,2a+b\}$ admits no
canonical-free colouring. At $p=cn^{-1/2}$, the expected number of these four-point
configurations is of constant order. Longer chains of the
same colour constraints also contribute; counting them gives
the full critical law.

\begin{theorem}[Schur critical limit]\label{thm:schur-intro}
For every $c>0$, there exists $F(c)>0$ such that
\[
 \Pr\bigl([n]_{cn^{-1/2}}\canarrow(1,1,-1)\bigr)
 \longrightarrow F(c).
\]
Moreover, $0<F(c)<1$ for $0<c<\sqrt{\pi/2}$, while $F(c)=1$ for
$c\ge\sqrt{\pi/2}$.
\end{theorem}

Theorem~\ref{thm:schur} gives $F$ explicitly and proves uniform
convergence. Thus $n^{-1/2}$ is still a threshold function,
although no fixed $c>0$ gives a $0$-statement at $p=cn^{-1/2}$.
The exception is Schur's equation itself: larger systems with a
Schur density core satisfy Theorem~\ref{thm:zero}.

We now turn to the $1$-statement. Passing from dense arithmetic
to sparse hosts is a broader theme, also present in Green and
Tao's work on linear equations in the primes~\cite{GreenTaoLinearPrimes}.
For random sets, Conlon and Gowers~\cite{CG} and
Schacht~\cite{Schacht} developed general transference methods.
Hypergraph containers~\cite{BMS,ST} offer another route, by
describing all sets that avoid the desired configurations.
For progressions, Allen, B\"ottcher, Lada and Mergoni
Cecchelli~\cite{ABLM} gave another proof of the canonical
upper bound using their generalised transference principle.

The progression proofs~\cite{AKMMO,ABLM} use the fact that
every set of positive density contains long progressions.
They find a monochromatic progression in a large colour class
using random Szemer\'edi or a dense model, and merge small
classes before seeking a rainbow progression. For general
partition-regular systems, even a positive-density set may
contain no solution: the odd integers contain no Schur triple.
Merging colours also creates new equalities, which may disappear
when the original colours are restored. We must therefore
preserve the equalities inside the prescribed blocks.

Applying Lefmann's theorem to high-dimensional Deuber
configurations gives the dense count needed for our reduction,
with constants independent of the palette. We encode the
selected canonical solutions as hyperedges on labelled points
and apply containers. Each container with few solutions forces
any supported colouring with small merged classes to put too
large a proportion of its points in one fixed set of integers.
Concentration then rules this out in the random sample,
simultaneously for all containers. This proves the upper bound
at the same scale.

\begin{theorem}[1-statement]\label{thm:one}
Let $A\in\Z^{\ell\times k}$ be an irredundant partition-regular
matrix with $k\ge3$ and $\rank(A)=\ell\ge1$. Fix a partition
$\cI$ witnessing the columns condition. There are $C,c>0$ such
that, for all sufficiently large $n$ and
$Cn^{-1/m(A)}\le p\le1$,
\[
 \Pr\bigl([n]_p\canarrow A\bigr)\ge1-e^{-cnp}.
\]
\end{theorem}

The constants may be chosen uniformly over the finitely many
columns-condition partitions. Table~\ref{tab:history} compares
these threshold results with their predecessors.

\subsection{Organization}
Section~\ref{sec:sparse-rank} proves extension, and
Section~\ref{sec:zero} applies it to the non-Schur $0$-statement.
Section~\ref{sec:schur} determines the Schur critical law, and
Section~\ref{sec:counting} proves the upper bound.
Appendices~\ref{app:finite},
\ref{app:three-variable-proofs} and~\ref{app:freiman} contain the
uniform canonical supersaturation argument, three-variable proofs
and Freiman application, respectively.

\begin{table}[h]
\centering
\footnotesize
\renewcommand{\arraystretch}{1}
\setlength{\extrarowheight}{0.5mm}
\begin{tabularx}{\textwidth}{@{}>{\raggedright\arraybackslash}p{0.24\textwidth}>{\raggedright\arraybackslash}p{0.13\textwidth}>{\raggedright\arraybackslash}p{0.22\textwidth}>{\raggedright\arraybackslash}X@{}}
\toprule
Configurations & Colours & Dense theorem & Random result\\
\midrule
$k$-term progressions & Fixed $r\ge2$ & van der Waerden~\cite{VDW}
& $n^{-1/(k-1)}$; R\"odl--Ruci\'nski~\cite{RRvdW}\\
$k$-term progressions & Unrestricted & Erd\H{o}s--Graham~\cite{EG}
& $n^{-1/(k-1)}$; AKMMO~\cite{AKMMO}\\
Partition-regular systems & Fixed $r\ge2$ & Rado~\cite{Rado}
& $n^{-1/m(A)}$; R\"odl--Ruci\'nski~\cite{RR},
Friedgut--R\"odl--Schacht~\cite{FRS}\\
Non-Schur partition-regular systems & Unrestricted & Lefmann~\cite{Lefmann}
& $n^{-1/m(A)}$; matching $0$- and $1$-statements
(Theorems~\ref{thm:one} and~\ref{thm:zero})\\
Schur's equation & Unrestricted & Lefmann~\cite{Lefmann}
& $n^{-1/2}$; critical curve $F(c)$ at $p=cn^{-1/2}$
(Theorem~\ref{thm:schur-intro})\\
\bottomrule
\end{tabularx}
\vspace{+0.1cm}
\caption{\footnotesize The Rado-type results.}
\label{tab:history}
\end{table}

\section{An extension theorem for sparse random sets}\label{sec:sparse-rank}
\label{subsec:zero-rank}

We first establish the freedom needed for the colouring construction.
All spaces, ranks and spans in this section are over a field $\mathbb F$.
For a finite set $S$, retain each point independently with probability
$p$ to obtain $S_p$. For $H\subseteq[s]=\{1,\ldots,s\}$, let $B[H]$
be the submatrix on columns $H$, and write $\overline J=[s]\setminus J$.

For a full-row-rank matrix $B\in\mathbb F^{r\times s}$ such that
deleting any two columns preserves rank, define
\[
 m(B)=\max_{\substack{J\subseteq[s]\\|J|\ge2}}
 \frac{|J|-1}{|J|-1-r+\rank B[\overline J]}.
\]
The two-column condition makes all denominators positive. This
definition agrees with~\eqref{eq:m(A)} over $\Q$.
We call $B$ \emph{strictly balanced} if
\[
 m(B)=\frac{s-1}{s-r-1}>
 \frac{|J|-1}{|J|-1-r+\rank B[\overline J]}
 \qquad(2\le |J|<s),
\]
and \emph{regular} if it is strictly balanced and deleting any two
columns preserves its rank. Our target in this section is the following theorem which is essentially Theorem~\ref{thm:rank-informal}.

\begin{theorem}\label{thm:sparse-rank}
Let $B\in\mathbb F^{r\times s}$ be regular, with $s\ge3$ and
$r\ge1$, and write $m(B)=a/b$ in lowest terms. There is $K_s>0$
such that, for every $N$-element set
$S\subseteq\mathbb F\setminus\{0\}$ with $N\ge1$ and
$0\le p\le(8s)^{-1}N^{-b/a}$, with probability at least
$1-K_sN^{-1/a}$ the following holds. Every $X\subseteq\Omega=S_p$
contains a set $Y\subseteq X$ with 
$$|Y|\ge |X|/m(B)$$
such that every map $g:Y\to\mathbb F$ extends to a function
$f:\Omega\to\mathbb F$ satisfying
\[
 B(x_1,\ldots,x_s)^{\mathsf T}=0
 \quad\Longrightarrow\quad
 B(f(x_1),\ldots,f(x_s))^{\mathsf T}=0
\]
for all pairwise distinct $x_1,\ldots,x_s\in\Omega$.
\end{theorem}

\subsection{Proof overview and notations}\label{subsec:solution-language}

The extension property is a bound on the freedom left by the
solution equations. Let $B\in\mathbb F^{r\times s}$ have full
row rank and let $\Omega\subseteq\mathbb F\setminus\{0\}$ be
finite. Write $V_B(\Omega)$ for the vector space of maps
$f:\Omega\to\mathbb F$ satisfying
\[
 B(x_1,\ldots,x_s)^{\mathsf T}=0
 \quad\Longrightarrow\quad
 B(f(x_1),\ldots,f(x_s))^{\mathsf T}=0
\]
for all pairwise distinct $x_1,\ldots,x_s\in\Omega$; the image
values may coincide. For $X\subseteq\Omega$, write
$V_B(\Omega)|_X=\{f|_X:f\in V_B(\Omega)\}$.
A set $Y\subseteq\Omega$ \emph{admits arbitrary prescriptions} if
$V_B(\Omega)|_Y=\mathbb F^Y$, where $\mathbb F^T$ is the space
of functions on $T$. Gaussian elimination shows that the largest
such subset of $X$ has size $\dim(V_B(\Omega)|_X)$. Our target
for regular $B$ is therefore
\begin{equation}\label{eq:zero-sparse-density}
 |X|\le m(B)\cdot\dim\bigl(V_B(\Omega)|_X\bigr)
 \qquad(X\subseteq\Omega).
\end{equation}
We prove this in three steps.

\medskip
\noindent\textbf{Step 1: Exclude dense diagrams
(Subsection~\ref{subsec:restricted-diagrams}).}
A \emph{solution copy} is an injective map $x:[s]\to\Omega$
satisfying $Bx=0$, with \emph{support} $E_x=x([s])$.
The \emph{incidence graph} of an indexed family
$(E_\lambda)_{\lambda\in\Lambda}$ on $T$ is the bipartite graph
on $T\mathbin{\dot\cup}\Lambda$ joining $u$ to $\lambda$ when
$u\in E_\lambda$; distinct indices remain separate even if
their supports agree. Without a specified ambient set, take
$T=\bigcup_\lambda E_\lambda$. A \emph{$B$-diagram} is an abstract
incidence graph of copy and variable vertices. Each copy meets
$s$ distinct variables, with incidences bijectively labelled by
$[s]$, prescribing the equations of $B$ on them. A
\emph{realization} assigns distinct host elements to the variables
and satisfies all these equations.

Lemma~\ref{lem:zero-diagrams} rules out connected diagrams of
bounded cycle rank whose equations leave insufficient dimension
for~\eqref{eq:zero-sparse-density}. This is the only probabilistic
step.

\medskip
\noindent\textbf{Step 2: Extend through each component
(Subsection~\ref{subsec:component-extension}).}
For a copy $x$, let $L_x\subseteq\mathbb F^\Omega$ be the span
of the rows of $B$ placed on $x$: a row $(a_1,\ldots,a_s)$
becomes $\sum_i a_i e_{x(i)}$, where $e_u$ is the standard
vector at $u$. Set $R_B(\Omega)=\sum_xL_x$, so
$V_B(\Omega)=R_B(\Omega)^\perp$ under the standard pairing.
Let $\mathcal H_B(\Omega)$ be the simple $s$-uniform
\emph{solution hypergraph} on $\Omega$ whose edges are the
distinct supports $E_x$. A \emph{solution component} $C$ is the point set of a
component, including isolated points. Different components
involve disjoint variables, hence
\begin{equation}\label{eq:component-freedom}
 V_B(\Omega)|_C=V_B(C),\qquad
 \dim\bigl(V_B(\Omega)|_C\bigr)=|C|-\dim R_B(C).
\end{equation}

An \emph{exploration} is a sequence of copies in which each
copy after the first meets earlier points; it is \emph{retained}
if each copy increases the span of the equations. Regularity
excludes zero columns, so Lemma~\ref{lem:zero-retained-exploration}
gives an exploration of any non-trivial component $C$, starting
with any prescribed copy, covering $C$ and spanning $R_B(C)$.
The excluded diagrams force
the rank increase to equal the rank on the new coordinates
(Lemma~\ref{lem:zero-exploration-growth}), so every permitted
assignment on the old points extends to the new ones
(Lemma~\ref{lem:zero-step-extension}). Thus every assignment in
$\ker B$ on a starting copy extends throughout $C$, and
$\dim V_B(C)\ge |C|/m(B)$
(Lemma~\ref{lem:zero-exploration}).

\medskip
\noindent\textbf{Step 3: Obtain freedom on every subset
(Subsection~\ref{subsec:subset-freedom}).}
A subset can inherit relations from outside it. Viewing
$\mathbb F^X$ as the vectors supported on $X$, orthogonality gives
\begin{equation}\label{eq:extension-supported-relations}
 \dim\bigl(V_B(\Omega)|_X\bigr)
 =|X|-\dim\bigl(R_B(\Omega)\cap\mathbb F^X\bigr).
\end{equation}
Since $m(B)=(s-1)/(s-r-1)$, our target is equivalently
\begin{equation}\label{eq:extension-relation-bound}
 \dim\bigl(R_B(\Omega)\cap\mathbb F^X\bigr)
 \le\frac{r}{s-1}|X|\qquad(X\subseteq\Omega).
\end{equation}
For coordinate restriction $\pi_J$ on $\ker B$, rank--nullity gives
\begin{equation}\label{eq:extension-projection-rank}
 \dim\pi_J\ker B=|J|-r+\rank B[\overline J].
\end{equation}
For $\varnothing\ne J\subseteq[s]$, the definition of $m(B)$
and the rank condition give
$\dim\pi_J\ker B\ge1+(|J|-1)/m(B)$.
By Step~2, these are also the local freedom bounds on each copy.
Combining them with the component
dimension bound, submodularity of restriction dimensions yields
the bound on every subset (Lemma~\ref{lem:zero-subsets}).
Isolated points are free, so summing over components proves
\eqref{eq:zero-sparse-density} and hence
Theorem~\ref{thm:sparse-rank}.

\subsection{Pseudo-random condition: restricted diagrams}\label{subsec:restricted-diagrams}

Fix $B$ as in Theorem~\ref{thm:sparse-rank}. Strict balance gives
\begin{equation}\label{eq:extension-strict-growth}
 (s-1)\rank B[H]-r|H|>0\qquad(1\le |H|\le s-2).
\end{equation}
Substitute $J=[s]\setminus H$ in the strict-balance inequality
and cross-multiply. In particular, every column is nonzero.

For a connected diagram with $e$ copies, $v$ variables and combined
equation rank $R$, the space of assignments satisfying its equations
has dimension $v-R$. Our target freedom bound asks that this dimension
be at least $v/m(B)$. Since $m(B)=(s-1)/(s-r-1)$, multiplying the
surplus $(v-R)-v/m(B)$ by $s-1$ gives $rv-(s-1)R$. We therefore set
\[
 j=(s-1)e+1-v,\qquad \Delta=rv-(s-1)R.
\]
Thus $\Delta\ge0$ is precisely the desired freedom bound for the
combined system. Its incidence graph has $e+v$ vertices and $se$
edges, so $j$ is its cycle rank. This count uses the separately
indexed copy vertices, before any repeated supports are merged.
We now exclude diagrams with $\Delta<0$ and bounded $j$.

\begin{lemma}\label{lem:zero-diagrams}
Let $B\in\mathbb F^{r\times s}$ have rank $r$, where
$1\le r\le s-2$. Write
$m=(s-1)/(s-r-1)=a/b$ in lowest terms, and set
$g=\gcd(r,s-1)$ and $D=(s-1)(r/g+1)$.
There is $K_s>0$ such that, for every $N$-element set
$S\subseteq\mathbb F\setminus\{0\}$ with $N\ge1$ and
$0\le p\le(8s)^{-1}N^{-b/a}$, with probability at least
$1-K_sN^{-1/a}$ no connected $B$-diagram with
$$j\le D\qquad \text{and}\qquad \Delta<0$$
has a realization in $S_p$.
\end{lemma}

\begin{proof}
\claimsof{lem:zero-diagrams}
By monotonicity, it suffices to take $p=c_sN^{-1/m}$, where
$c_s=(8s)^{-1}$. We first bound the number of diagrams.

\begin{claim}\label{cl:extension-diagram-shapes}
For fixed $e$ and $j$, the number of connected $B$-diagrams,
up to isomorphism preserving the coordinate labels, is at most
$C_s^e(se+1)^{2j}$, where $C_s=4^s s!$.
\end{claim}

\begin{poc}
Choose a spanning tree of the incidence graph. There are exactly
$j$ incidences outside the tree. Split each of them at its variable
endpoint, replacing that endpoint by a new variable leaf. The
resulting graph is a tree on $e+v+j=se+1$ vertices, and each copy
still has degree $s$.
Root this tree at a variable vertex and order the children at each
vertex. There are at most $4^{se}$ rooted ordered tree shapes with
$se$ edges. Depth specifies the two vertex classes, and assigning
coordinate roles at the copies contributes at most $(s!)^e$.
To recover the original diagram, record the $j$ pairs consisting
of a detached leaf and the variable vertex with which it is
identified. There are at most $(se+1)^{2j}$ choices.
Every diagram has such an encoding; multiple encodings of the
same diagram only increase the upper bound.
\end{poc}

A diagram of rank $R$ on $v$ variables has at most $N^{v-R}$
realizations in $S$: choosing free coordinates determines all others.
Each uses $v$ distinct points, so its expected number of realizations
in $S_p$ is at most
\[
 N^{v-R}p^v
 =c_s^vN^{\,v(1-1/m)-R}
 =c_s^vN^{\Delta/(s-1)},
\]
where $1-1/m=r/(s-1)$.

The integer $\Delta=rv-(s-1)R$ is divisible by $g$. Also,
$\gcd(s-1,s-r-1)=\gcd(s-1,r)=g$, so $a=(s-1)/g$.
Consequently $\Delta<0$ implies
$\Delta/(s-1)\le-g/(s-1)=-1/a$.
Using $v=(s-1)e+1-j$ and Claim~\ref{cl:extension-diagram-shapes},
a union bound over all bad diagrams gives failure probability at most
\[
 N^{-1/a}\sum_{j=0}^{D}c_s^{1-j}
 \sum_{e\ge1}(C_sc_s^{s-1})^e(se+1)^{2j}.
\]
Indeed,
$C_sc_s^{s-1}=2/3$ for $s=3$, while for $s\ge4$,
 $C_sc_s^{s-1}=\frac{4s!}{(2s)^{s-1}}
 \le2^{3-s}\le\frac12.$
For each $j\le D$, the inner series is therefore convergent.
Since $D\le(s-1)^2$, the whole sum is bounded by a constant
$K_s$ depending only on $s$.
\end{proof}

\subsection{Extending values through a solution component}\label{subsec:component-extension}

We first record an observation that requires only the absence
of zero columns.

\begin{lemma}\label{lem:zero-retained-exploration}
Let $B\in\mathbb F^{r\times s}$ have no zero column, and let
$\Omega\subseteq\mathbb F\setminus\{0\}$ be finite. Every
non-trivial solution component $C$ has a retained exploration
starting with any prescribed copy, covering $C$, and spanning
$R_B(C)$.
\end{lemma}

\begin{proof}
Order all copies of $C$, beginning with the prescribed one, so
that every later copy meets earlier points. Retain a copy precisely
when its row space is not already spanned by the retained rows.
A discarded copy introduces no new point: since $B$ has no zero
column, some placed row would have a nonzero coefficient at any
such point, while every earlier row would be zero there.
Inductively, retained copies cover every point encountered so far.
Each later retained copy therefore meets the previously retained
points, even if its original intersection was with a discarded
copy. The retained copies thus form an exploration covering $C$
and spanning all its equations.
\end{proof}

We now work deterministically with a finite set $\Omega$
satisfying the conclusion of Lemma~\ref{lem:zero-diagrams}:
$\Omega$ realizes no connected solution diagram with 
$$j\leq D\qquad \text{and}\qquad \Delta<0.$$ 

\begin{lemma}\label{lem:zero-exploration}
Let $B$ be regular of rank $r\ge1$ with $s\ge3$ columns, and set
$g=\gcd(r,s-1)$ and $D=(s-1)(r/g+1)$.
Suppose that a finite $\Omega\subseteq\mathbb F\setminus\{0\}$
realizes no connected $B$-diagram with $j\le D$ and $\Delta<0$.
For every non-trivial solution component $C$, the space
$V=V_B(\Omega)|_C$ satisfies
\[
 (s-1)\dim V\ge(s-r-1)|C|.
\]
Moreover, for every solution copy $x$ with support $E\subseteq C$,
\[
 V|_E=\ker B,
\]
where restriction to $E$ is ordered by $x$.
\end{lemma}

Starting with arbitrary values in $\ker B$ on one copy, we use
the forbidden-diagram condition to extend them through the component.

Throughout the rest of this subsection, let $B$ and $\Omega$ satisfy
the hypotheses of Lemma~\ref{lem:zero-exploration}.
By~\eqref{eq:extension-strict-growth}, $B$ has no zero column,
so Lemma~\ref{lem:zero-retained-exploration} supplies the required
retained exploration. After $i$ such copies,
let $F_i$ be the set of covered points, let
$R_i\subseteq\mathbb F^{F_i}$ be the span of their rows, and write
\[
 V_i=R_i^\perp,\qquad
 \Delta_i=r|F_i|-(s-1)\dim R_i.
\]
Earlier rows are extended by zero when the point set grows.
Thus $V_i$ consists of the assignments satisfying all equations
imposed so far. For $i\ge2$, let $H_i\subseteq[s]$ be the new
coordinate positions of the $i$th copy, and set
\[
 h_i=|H_i|,\qquad \alpha_i=\dim R_i-\dim R_{i-1}.
\]
Adding this copy brings $h_i$ new unknown values and $\alpha_i$
new independent equations. On the new coordinates, its equations
already have rank $\rank B[H_i]$. We shall prove that
$\alpha_i$ is exactly this rank, leaving no further restriction on
the old values.

Track $\Delta_i$, starting at $\Delta_1=r$. Attaching a copy at
one old point leaves $\Delta$ unchanged; further overlap decreases
it. If $\Delta$ first became negative, only boundedly many such
overlaps could have occurred, giving a forbidden diagram.

\begin{lemma}\label{lem:zero-exploration-growth}
Every retained exploration of a non-trivial solution component
satisfies $0\le\Delta_i\le r$ for every $i$, and
\[
 1\le h_i\le s-1,\qquad \alpha_i=\rank B[H_i]
 \qquad(i\ge2).
\]
\end{lemma}

\begin{proof}
Every prefix is a connected diagram, since each copy meets earlier
points. The first copy has $s$ points and row rank $r$, so
$\Delta_1=rs-(s-1)r=r$.
For every later retained addition, set
\[
 w_i=(s-1)\alpha_i-rh_i,
 \qquad \Delta_i=\Delta_{i-1}-w_i.
\]
Let $L_i$ be the row space of the new copy. Then
$R_i=R_{i-1}+L_i$, and restriction to the new coordinates sends
$L_i$ onto $\operatorname{row}(B[H_i])$.
Earlier rows vanish there, so this restriction induces a surjection
\[
 R_i/R_{i-1}\longrightarrow\operatorname{row}(B[H_i]).
\]
Hence $\rank B[H_i]\le\alpha_i\le r$.

Because the new copy meets earlier points, $0\le h_i\le s-1$.
If $h_i=s-1$, deleting the one old column preserves rank, so
$\alpha_i=r$ and $w_i=0$.
If $1\le h_i\le s-2$,~\eqref{eq:extension-strict-growth} gives
$w_i\ge(s-1)\rank B[H_i]-rh_i>0$.
If $h_i=0$, retention gives $\alpha_i\ge1$, and hence
$w_i\ge s-1$. Thus every $w_i$ is nonnegative, and
$w_i=0$ occurs exactly when $h_i=s-1$.
Since $g$ divides both $r$ and $s-1$, every positive $w_i$
is at least $g$.

Suppose that some prefix has negative $\Delta$, and take the first
such prefix, ending at step $i$. Before its last addition,
\[
 \sum_{\nu=2}^{i-1}w_\nu=r-\Delta_{i-1}\le r.
\]
There are therefore at most $r/g+1$ additions with positive $w$,
including the last one. The incidence cycle rank of the prefix is
$j_i=(s-1)i+1-|F_i|$.
It starts at zero, and adjoining a copy with $h_\nu$ new points
increases it by $s-1-h_\nu$.
This increase is zero when $w_\nu=0$, and at most $s-1$ otherwise.
Consequently the first negative prefix has
$j_i\le(s-1)(r/g+1)=D$.
Its variables already have distinct values in $\Omega$, and
its combined equation rank is $\dim R_i$. It is therefore a
realization of a forbidden diagram, contradicting the hypothesis.
We conclude that
\[
 0\le\Delta_i\le r,\qquad \sum_{i\ge2}w_i\le r.
\]

The two-column condition gives $r\le s-2$. Thus $h_i=0$ is
impossible, since it would give $w_i\ge s-1>r$.
For $1\le h_i\le s-1$, we have
$(s-1)\rank B[H_i]-rh_i\ge0$.
If $\alpha_i>\rank B[H_i]$, integrality would give
\[
 w_i\ge(s-1)\rank B[H_i]-rh_i+(s-1)>r,
\]
again impossible. This proves the lemma.
\end{proof}

The rank equality now gives the required extension.

\begin{lemma}\label{lem:zero-step-extension}
At a step $i\ge2$ of a retained exploration, if
$\alpha_i=\rank B[H_i]$, then restriction $V_i\to V_{i-1}$ is
surjective. Thus every assignment in $V_{i-1}$ extends to one in $V_i$.
\end{lemma}

\begin{proof}
The restriction map $V_i\to V_{i-1}$ is well defined because
$R_{i-1}\subseteq R_i$.
Its kernel consists of assignments zero on every old point.
The old equations then vanish, and the new equations become
$B[H_i]z=0$. Hence this kernel has dimension
$h_i-\rank B[H_i]$.
On the other hand, rank--nullity for the full systems gives
\[
 \dim V_i
 =|F_i|-\dim R_i
 =\dim V_{i-1}+h_i-\alpha_i.
\]
Subtracting the kernel dimension, and using
$\alpha_i=\rank B[H_i]$, shows that the image has dimension
$\dim V_{i-1}$. This proves surjectivity: the new equations impose
no new restriction on the old values.
\end{proof}

\begin{proof}[Proof of Lemma~\ref{lem:zero-exploration}]
Fix a non-trivial solution component $C$ and a solution copy $x$
with support $E\subseteq C$. By Lemma~\ref{lem:zero-retained-exploration},
choose a retained exploration starting with $x$, covering $C$, and spanning
all its equations. Let $t$ be its number of copies.
Since $F_t=C$ and $R_t=R_B(C)$, equation~\eqref{eq:component-freedom}
gives $V=V_B(\Omega)|_C=V_t$.

The first copy has $V_1=\ker B$ in its coordinate order.
Lemmas~\ref{lem:zero-exploration-growth}
and~\ref{lem:zero-step-extension} show that every assignment in
$V_1$ extends successively to $V_t$. Conversely, every member of
$V_t$ satisfies the equations of the starting copy $x$.
Thus $V|_E=\ker B$, with coordinates ordered by $x$.
As $x$ was arbitrary, this holds on every copy. Finally, Lemma~\ref{lem:zero-exploration-growth} gives
\[
 \begin{aligned}
 (s-1)\dim V-(s-r-1)|C|
 =r|C|-(s-1)\dim R_t=\Delta_t\ge0,
 \end{aligned}
\]
which proves the dimension bound.
\end{proof}

\subsection{Freedom on every subset}\label{subsec:subset-freedom}

We now pass to arbitrary subsets, including those containing no
whole copy, by combining local freedom, the component bound and
submodularity.

\begin{lemma}\label{lem:zero-subsets}
Let $B$ be regular of rank $r\ge1$ with $s\ge3$ columns, and set
$g=\gcd(r,s-1)$ and $D=(s-1)(r/g+1)$.
If a finite $\Omega\subseteq\mathbb F\setminus\{0\}$ realizes no
connected $B$-diagram with $j\le D$ and $\Delta<0$, then
\[
 |X|\le m(B)\cdot\dim\bigl(V_B(\Omega)|_X\bigr)\qquad\text{for every $X\subseteq\Omega$}.
\]
\end{lemma}

\begin{proof}
Fix a non-trivial solution component $C$ and write
$V=V_B(\Omega)|_C$. By Lemma~\ref{lem:zero-exploration},
$V|_E=\ker B$ for every solution copy $x$ with support
$E\subseteq C$, with coordinates ordered by $x$.
Thus $\dim(V|_E)=s-r$. Moreover, the definition of $m(B)$ gives
\[
 (s-1)\dim(V|_J)-(s-r-1)|J|\ge r
 \qquad(\varnothing\ne J\subseteq E).
\]
For $|J|\ge2$, this is the defining density inequality after
rearrangement; for a singleton, the projection has dimension one
because deleting its column preserves rank.

Restriction dimensions are submodular: for $X,E\subseteq C$,
\[
 \dim(V|_X)+\dim(V|_E)
 \ge\dim(V|_{X\cup E})+\dim(V|_{X\cap E}).
\]
These are ranks of coordinate evaluation functionals, so the
inequality follows from the dimension formula for their spans. If the support $E$ of a solution copy meets $X$, apply
the preceding local bound to $J=X\cap E$. Since
$(s-1)(s-r)-(s-r-1)s=r$, submodularity yields
\[
 \begin{aligned}
 (s-1)\dim(V|_X)-(s-r-1)|X|\ge(s-1)\dim(V|_{X\cup E})-(s-r-1)|X\cup E|.
 \end{aligned}
\]
Starting from nonempty $X$, connectivity lets us add meeting copies
until they cover $C$. Iterating gives
\[
 \begin{aligned}
 (s-1)\dim(V|_X)-(s-r-1)|X|
 \ge(s-1)\dim V-(s-r-1)|C|\ge0,
 \end{aligned}
\]
where the last inequality is Lemma~\ref{lem:zero-exploration}.
This proves the bound on $C$, including the empty subset.

Equations from different components involve disjoint variables,
so their permitted assignments are independent and restriction
dimensions add. An isolated point has one freely assignable value.
Summing the component bounds, and using $m(B)>1$ for isolated points,
proves the result for every $X\subseteq\Omega$.
\end{proof}

\begin{proof}[Proof of Theorem~\ref{thm:sparse-rank}]
The two-column condition gives $1\le r\le s-2$, and regularity
gives $m(B)=(s-1)/(s-r-1)=a/b$. Set
$g=\gcd(r,s-1)$ and $D=(s-1)(r/g+1)$.
Lemma~\ref{lem:zero-diagrams} shows that, outside an event of
probability at most $K_sN^{-1/a}$, the sample $\Omega=S_p$
realizes no connected diagram with $j\le D$ and $\Delta<0$.
On this event, Lemma~\ref{lem:zero-subsets} proves
\eqref{eq:zero-sparse-density} simultaneously for every subset.

Choose free coordinate positions $Y\subseteq X$ for
$V_B(\Omega)|_X$. Every map $Y\to\mathbb F$ then extends to
$V_B(\Omega)$, and $|Y|=\dim(V_B(\Omega)|_X)\ge |X|/m(B)$,
proving the claimed extension. Taking $c_s=(8s)^{-1}$ also
proves Theorem~\ref{thm:rank-informal}.
\end{proof}

\section{The 0-statement: non-Schur systems}\label{sec:zero}\label{sec:two-zero}

We now turn extension into a canonical-free colouring, proving the
non-Schur part of Theorem~\ref{thm:zero}. Ordinary two-colourings
handle one or at least three prescribed blocks
(Proposition~\ref{prop:other-zero}). For two blocks, we pass to a
density core and use vector partition to coordinate all colours.

\begin{proposition}\label{prop:other-zero}
Let $A\in\Z^{\ell\times k}$ be an irredundant partition-regular matrix
with $k\ge3$ and $\rank(A)=\ell\ge1$. Suppose that the prescribed
columns-condition partition $\mathcal I$ has one or at least three blocks.
There is $c_0>0$ such that, for every sequence
$p=p(n)\in[0,1]$ with $p\le c_0n^{-1/m(A)}$,
\[
 \Pr\bigl([n]_p\canarrow A\bigr)=o(1).
\]
\end{proposition}

\begin{proof}
By the random Rado $0$-statement of R\"odl and Ruci\'nski~\cite{RR},
with the corrected proof of Hancock and Treglown~\cite[Section~3]{HT},
for a sufficiently small $c_0>0$ the set $[n]_p$ has, with high
probability, a two-colouring with no monochromatic proper $A$-solution.
Since $k\ge3$, this colouring has no rainbow proper solution either.
A one-block pattern is monochromatic, while a pattern with at least
three blocks needs at least three colours. Thus this colouring is
canonical-free.
\end{proof}

Write $\mathcal P_2(A)$ for the ordered two-block partitions
witnessing the columns condition for $A$. We prove a uniform
statement on finite sets of positive real numbers: one colouring
avoids every pattern in $\mathcal P_2(A)$, as well as monochromatic
and rainbow solutions.

\begin{proposition}\label{prop:two-zero}
Let $A\in\Z^{\ell\times k}$ be irredundant and partition regular, with
$k\ge3$, $\rank A=\ell\ge1$, and $\mathcal P_2(A)\ne\varnothing$.
Assume $A$ is not a non-zero scalar and column-permutation copy of $(1,1,-1)$.
Write $m(A)=a/b$, where $a,b$ are coprime positive integers.
There are $c_k, K_k>0$ such that,
for every $S\subset\R_{>0}$ of size $N\ge1$ and
$0\le p\le c_kN^{-b/a}$, with probability at least $1-K_kN^{-1/a}$ one
colouring of $S_p$ has no proper canonical $A$-solution for any
$\cI\in\mathcal P_2(A)$.
\end{proposition}

For the two-block argument, fix $A$ and $m(A)=a/b$ as in
Proposition~\ref{prop:two-zero}. We project $A$ onto a regular
density core $B$ with $m(B)=m(A)$, so every proper $A$-solution
contains a proper core solution. The extension theorem supplies
freedom on the core, and Edmonds' vector partition theorem turns
this into a colouring. Three cases follow directly; three-variable
cores require more information about their overlaps.

\subsection{Density cores}

The required rank property follows from partition regularity.
Set $\mathcal V=\ker_{\Q}A$. For $W\subseteq[k]$, the projection
identity~\eqref{eq:extension-projection-rank} gives
\begin{equation}\label{eq:projection-rank}
 \dim\pi_W(\mathcal V)=|W|-\ell+\rank(A_{\overline W}).
\end{equation}

\begin{lemma}\label{lem:rank}
Let $A\in\Z^{\ell\times k}$ be an irredundant partition-regular
matrix with $k\ge3$ columns and full row rank $\ell\ge1$.
Then deleting at most two columns of $A$ leaves a matrix of rank $\ell$.
Consequently, $k-\ell\ge2$, every denominator in the definition
of $m(A)$ in~\eqref{eq:m(A)} is positive, and $m(A)>1$.
\end{lemma}

\begin{proof}
Write $A=(\mathbf a_1,\ldots,\mathbf a_k)$ and set
$\mathcal V=\ker_{\Q}A$. By partition regularity, fix a partition
$\cI=(I_0,\ldots,I_s)$ witnessing the columns condition.
Suppose that deleting at most two columns lowers
rank. There is a non-zero rational row vector $\mathbf u$ such
that $c_i=\mathbf u\mathbf a_i$ is non-zero in at most two
positions. Let $I_r$ be the first block of the fixed
columns-condition partition containing a non-zero $c_i$.
Applying $\mathbf u$ to the columns condition gives
$\sum_{i\in I_r}c_i=0$, since all earlier $c_i$ vanish.
The support therefore consists of two positions $i,j$ in
$I_r$, with $c_i=-c_j\ne0$. Every vector in $\mathcal V$ satisfies
$x_i=x_j$, contrary to irredundancy.

Rank preservation implies $k-\ell\ge2$. For $|W|\ge2$, adding
back $|W|-2$ columns to $A_{\overline W}$ gives a matrix
of rank $\ell$, so
$\rank(A_{\overline W})\ge\ell-|W|+2$. The denominator
$|W|-1+\rank(A_{\overline W})-\ell$ defining $m(A)$ is
therefore positive. Taking $W=[k]$ gives
$m(A)\ge(k-1)/(k-\ell-1)>1$, since $\ell\ge1$.
\end{proof}

A \emph{density core} of $A$
is a strictly balanced full-row-rank integer matrix $B$, with
coordinates indexed by $J\subseteq[k]$, such that
\[
 \ker_{\Q}B=\pi_J(\ker_{\Q}A),\qquad m(B)=m(A).
\]
Projection preserves proper solutions. The next lemma makes the
core regular and records the induced one- or two-block partitions.

\begin{lemma}\label{lem:zero-core}
The matrix $A$ has a density core $B\in\Z^{r\times s}$
with $s\ge3$ and $r\ge1$. The matrix $B$ is regular, irredundant
and partition regular, and has no zero column. If its coordinates are indexed by $J\subseteq[k]$,
every $\mathcal I=(I_0,I_1)\in\mathcal P_2(A)$ induces a columns-condition
partition on $B$ by deleting empty sets from $(I_0\cap J,I_1\cap J)$.
\end{lemma}

\begin{proof}
\claimsof{lem:zero-core}
Choose $J$ inclusion-minimal among the sets attaining the maximum in
(\ref{eq:m(A)}). Set $s=|J|$ and choose a full-row-rank rational
matrix with kernel $\pi_J(\ker A)$. Clearing denominators gives an
integer matrix $B$ of rank $r$, so $\dim\pi_J\ker A=s-r$.

\begin{claim}\label{cl:zero-core-regular}
The matrix $B$ is a regular, irredundant, partition-regular density
core of $A$, with $s\ge3$, $r\ge1$ and no zero column.
\end{claim}

\begin{poc}
Monochromatic solutions and a proper positive solution of $A$
project to $B$, so $B$ is partition regular and irredundant.

For $H\subseteq J$, projecting first to $J$ and then to $H$ is the same
as projecting directly to $H$, so $\dim\pi_H\ker B=\dim\pi_H\ker A$. The maximum
defining $m(B)$ is therefore at most $m(A)$, and equality holds at all
$s$ coordinates. No proper subset with at least two coordinates can
attain equality, by the minimality of $J$. This proves strict balance
and $m(B)=m(A)$. The two-coordinate rank property gives $s\ge3$ and
$s-r\ge2$, while $m(A)>1$ gives $r\ge1$. The equality
$\dim\pi_{\{i,j\}}\ker B=2$ also shows that deleting any two columns
preserves rank. Hence $B$ is regular, and
\eqref{eq:extension-strict-growth}, applied to a singleton, shows
that every column is nonzero.
\end{poc}

\begin{claim}\label{cl:zero-core-partitions}
Every $\mathcal I=(I_0,I_1)\in\mathcal P_2(A)$ induces a
columns-condition partition on $B$ after its empty intersections
with $J$ are removed.
\end{claim}

\begin{poc}
Fix $\mathcal I=(I_0,I_1)\in\mathcal P_2(A)$.
The columns condition gives $\mathbf1_{I_0}\in\ker A$ and a vector
$w\in\ker A$ with $w_i=1$ for every $i\in I_1$. Write
$J_t=I_t\cap J$, and let $B_i$ denote column $i$ of $B$.
If both intersections are nonempty, projection of the
first vector gives $\sum_{i\in J_0}B_i=0$, and projection of the second
gives
\[
 \sum_{i\in J_1}B_i=-\sum_{i\in J_0}w_iB_i
 \in\Span_{\Q}\{B_i:i\in J_0\}.
\]
Thus $(J_0,J_1)$ witnesses the columns condition for $B$.
If $J_1$ is empty, the first projected vector is the all-ones vector;
if $J_0$ is empty, the second is. In either case the remaining block
has column sum zero, as required.
\end{poc}

The construction used only $A$, so the same core works for every
partition in $\mathcal P_2(A)$.
\end{proof}

\subsection{Colourings from extension and vector partition}

Fix a density core $B$ from Lemma~\ref{lem:zero-core}. Over $\R$,
we now turn the freedom of $V_B(U)$ into a colouring of a finite
$U\subseteq\R_{>0}$. The argument is deterministic until the final
sampling step.

\begin{proposition}\label{prop:zero-colourings}
Let $B$ be a density core of $A$, and let $U\subseteq\R_{>0}$ be
finite. Suppose that every $X\subseteq U$ contains a set
$Y\subseteq X$ with $|Y|\ge |X|/m(B)$ on which every prescribed
function extends to a member of $V_B(U)$, or equivalently that
\[
 |X|\le m(B)\cdot\dim(V_B(U)|_X)\qquad\text{for all $X\subseteq U$.}
\]
Then one colouring of $U$ avoids monochromatic and rainbow proper
$A$-solutions and every block pattern in $\mathcal P_2(A)$.
\end{proposition}
To express freedom as linear independence, choose a basis
$f_1,\ldots,f_D$ of $V_B(U)$ and associate to each $u\in U$ the vector
\[
 v_u=(f_1(u),\ldots,f_D(u))^{\mathsf T}\in\R^D.
\]
The evaluation map on $X$ is represented by the transpose of the
matrix with columns $(v_u)_{u\in X}$, so
\begin{equation}\label{eq:colour-evaluation-rank}
 \dim\Span\{v_u:u\in X\}=\dim(V_B(U)|_X).
\end{equation}
Thus $X$ admits arbitrary prescriptions exactly when its evaluation
vectors are linearly independent, regardless of the basis chosen.
Since $u\mapsto u$ belongs to $V_B(U)$, these vectors are nonzero
and distinct.

The rank bound for every subset now gives the required partition
into sets admitting arbitrary prescriptions, by Edmonds' theorem.

\begin{theorem}[Edmonds~\cite{Edmonds}]\label{thm:vector-partition}
Let $\mathbb F$ be a field, let $U$ be finite, and let
$(v_u)_{u\in U}$ be a family of vectors in an $\mathbb F$-vector space.
For an integer $t\ge1$, there is a partition
\[
 U=I_1\mathbin{\dot\cup}\cdots\mathbin{\dot\cup}I_t
\]
into possibly empty sets such that each family $(v_u)_{u\in I_j}$
is linearly independent if and only if
\[
 |X|\le t\cdot\dim\Span\{v_u:u\in X\}
 \qquad\text{for all $X\subseteq U$}.
\]
\end{theorem}

By~\eqref{eq:colour-evaluation-rank}, $U$ partitions into
$\lceil m(B)\rceil$ sets admitting arbitrary prescriptions,
separately on each part; prescriptions need not extend jointly.

Let $s$ and $r$ be the number of columns and rank of $B$.
A proper core solution on $E$ imposes $r$ independent equations, so
\begin{equation}\label{eq:colour-solution-freedom}
 \dim(V_B(U)|_E)\le s-r.
\end{equation}
If $I$ admits arbitrary prescriptions, then
$$|I\cap E|=\dim(V_B(U)|_{I\cap E})\le\dim(V_B(U)|_E)\le s-r.$$
Thus no part contains a core solution. We distinguish four cases
according to $m=m(B)=(s-1)/(s-r-1)$ and the rank, using stronger
restrictions on how solutions meet the parts.
\begin{enumerate}[leftmargin=*,itemsep=0.35em]
\item If $m>2$, Lemma~\ref{lem:zero-dense-core} partitions $U$
into fewer than $k$ sets admitting arbitrary prescriptions and
gives each part one colour. The columns condition ensures that
at least one prescribed block cannot fit into a single part.
\item If $m\le2$ and $r\ge2$, Lemma~\ref{lem:zero-higher-rank}
colours one part uniformly and the other with distinct new colours,
giving each core solution a repeated colour and at least three colours.
\item If $m\le2$, $r=1$ and $s\ge4$,
Lemma~\ref{lem:zero-long-equation} repeats the evaluation vectors
to obtain three colour classes whose pairwise unions admit
arbitrary prescriptions. Every core solution uses all three colours.
\item If $m\le2$, $r=1$ and $s=3$,
Lemma~\ref{lem:zero-ternary-core} uses freedom to describe and
colour the solution components.
\end{enumerate}

The first three cases use vector partition directly. We treat
three-variable cores in the next subsection.

\begin{lemma}\label{lem:zero-dense-core}
Let $B$ be a density core of $A$ with $m(B)>2$, and let
$U\subseteq\R_{>0}$ be finite. If $|X|\le m(B)\dim(V_B(U)|_X)$ for every
$X\subseteq U$, then one colouring of $U$ avoids monochromatic and
rainbow proper $A$-solutions and every block pattern in $\mathcal P_2(A)$.
\end{lemma}

\begin{proof}
Let $s$ and $r$ be the number of columns and rank of $B$, and set
$m=m(B)$. By Theorem~\ref{thm:vector-partition},
partition $U$ into $t=\lceil m\rceil\le s-1<k$ sets admitting
arbitrary prescriptions, and give them distinct colours.
Equation~\eqref{eq:colour-solution-freedom} excludes a monochromatic
core solution, and fewer than $k$ colours exclude a rainbow $A$-solution.

It remains to exclude the block patterns. By
Lemma~\ref{lem:zero-core}, each $\mathcal I\in\mathcal P_2(A)$
induces one or two nonempty blocks on $B$.
One block would give a monochromatic core solution, which has
already been excluded. Suppose that a solution has an induced
two-block pattern $(J_0,J_1)$. Each block lies in a colour class
admitting arbitrary prescriptions, and its possible values lie in
$\pi_{J_i}\ker B$. Hence $\dim\pi_{J_i}\ker B=|J_i|$ for $i=0,1$.
The projection formula gives $\rank B[J_0]=\rank B[J_1]=r$.
Since the columns on $J_0$ sum to zero, $|J_0|\ge r+1$, while
$|J_1|\ge r$. Thus $s\ge2r+1$, which implies $m\le2$,
a contradiction.
\end{proof}

When $m\le2$, the partition has two parts. Rank at least two
forces each core solution to occupy at least two points in each.

\begin{lemma}\label{lem:zero-higher-rank}
Let $B$ be a density core of $A$ with $m(B)\le2$ and $\rank B\ge2$,
and let $U\subseteq\R_{>0}$ be finite. If $|X|\le m(B)\dim(V_B(U)|_X)$
for every $X\subseteq U$, then one colouring of $U$ avoids
monochromatic and rainbow proper $A$-solutions and every block
pattern in $\mathcal P_2(A)$.
\end{lemma}

\begin{proof}
Set $r=\rank B$ and let $s$ be the number of columns.
Theorem~\ref{thm:vector-partition} gives a partition $U=P\dot\cup Q$
such that both sets admit arbitrary prescriptions.
For every core-solution support $E$, each of $E\cap P$ and
$E\cap Q$ therefore has size at most $s-r$, by
\eqref{eq:colour-solution-freedom}. Since $|E|=s$, this gives
\[
 |E\cap P|\ge r\ge2,\qquad |E\cap Q|\ge r\ge2.
\]
Colour $P$ with one colour, and give every point of $Q$ its own
new colour. Every core solution has a repeated colour from $P$
and at least two distinct new colours from $Q$. Thus every full
$A$-solution has a repeated colour and at least three colours,
excluding all required patterns.
\end{proof}

For rank-one cores with at least four variables, we instead apply
Edmonds' theorem to two copies of each evaluation vector.

\begin{lemma}\label{lem:zero-long-equation}
Let $B$ be a density core of $A$ of rank one with $s\ge4$ columns,
and let $U\subseteq\R_{>0}$ be finite. If $|X|\le m(B)\dim(V_B(U)|_X)$
for every $X\subseteq U$, then one colouring of $U$ avoids
monochromatic and rainbow proper $A$-solutions and every block
pattern in $\mathcal P_2(A)$.
\end{lemma}

\begin{proof}
Here $m=m(B)=(s-1)/(s-2)\le3/2$. On the labelled set
$U\times\{1,2\}$, assign to each $(u,j)$ the same vector $v_u$
defined above. If a set $Z$ of labelled copies has underlying
point set $X\subseteq U$, then
\[
 |Z|\le2|X|\le2m\dim(V_B(U)|_X)
 \le3\dim\Span\{v_u:(u,j)\in Z\}.
\]
Theorem~\ref{thm:vector-partition} partitions these labelled copies
into three sets with linearly independent vectors. The two copies
of a point lie in different sets, since their vectors are equal
and nonzero. Colour $u$ by the index of the set missing its copies.

Any two colour classes admit arbitrary prescriptions jointly.
Indeed, if their colours are $a,b$, every point in their union
has one copy in the third set. Their evaluation vectors are
therefore linearly independent. By
\eqref{eq:colour-solution-freedom}, a core solution cannot lie
in the union of two classes, so it uses all three colours.
This excludes monochromatic solutions and prescribed two-block
patterns. Since only three colours are used and $k\ge s\ge4$,
there is no rainbow $A$-solution either.
\end{proof}

\subsection{Three-variable cores}\label{subsec:zero-ternary}

For a ternary rank-one core, $m(B)=2$, but the two-part partition
does not control which pair in a triple receives a repeated colour.
We therefore use the overlap structure of the triples.

\begin{lemma}\label{lem:zero-ternary-core}
Let $B$ be a density core of $A$ of rank one with three
columns, and let $U\subseteq\R_{>0}$ be finite. If
$|X|\le2\dim(V_B(U)|_X)$ for every $X\subseteq U$, then one colouring of
$U$ avoids monochromatic and rainbow proper $A$-solutions and every
block pattern in $\mathcal P_2(A)$.
\end{lemma}

Zero coefficient sum is still handled by vector partition.
Otherwise the next two lemmas, proved in
Appendix~\ref{app:three-variable-proofs}, describe and colour the
components. For $k>3$, rainbow core triples in three colours
exclude full canonical solutions; for $k=3$, we must repeat a
colour while separating the prescribed first block. A larger
non-Schur system may have a Schur core, so both cases are needed.

In the non-zero-sum case, an induced two-block partition from
Lemma~\ref{lem:zero-core} starts with opposite coefficients, since
none is zero. Orienting this pair and rescaling gives $z=x+qy$,
where $q\in\Q_{>0}$.
We therefore fix such a $q$ and a finite set $U\subseteq\R_{>0}$,
and set $B=(-1,-q,1)$. To record the coefficients on each solution
support, write
\[
 R(x,y,z)=e_z-e_x-qe_y\in\R^U
\]
for each proper solution $z=x+qy$ in $U$.
Let $\mathcal R_q(U)$ be the set of these vectors, with identical
vectors identified. They are the placed rows of $B$, each with
support $\{x,y,z\}$, so $\Span\mathcal R_q(U)=R_B(U)$.
For the geometry of overlaps, write
$\mathcal H_q(U)=\mathcal H_B(U)$ for the simple support hypergraph
of this system. Its edges record solution supports, while
$\mathcal R_q(U)$ retains all distinct relations on them.

We call a hypergraph an \emph{incidence tree} if its incidence graph
is a tree, and an \emph{incidence unicycle} if that graph is connected
and has exactly one cycle. Attaching incidence trees to a hypergraph
means adjoining triples one at a time, each meeting the previous
vertices in exactly one point and introducing two new points.

The final component type is built from a \emph{loose triangle}:
three triples $\{a,b,u\}$, $\{b,c,v\}$ and $\{c,a,w\}$ on six distinct
vertices. Its \emph{private vertices} are $u,v,w$, each belonging to
just one of the three triples. Adding $\{u,v,w\}$ gives a
\emph{Pasch configuration}. Equivalently, label its vertices $x_{ij}$,
$1\le i<j\le4$, and its four triples
\[
 \begin{aligned}
 T_1&=\{x_{12},x_{13},x_{14}\},&
 T_2&=\{x_{12},x_{23},x_{24}\},\\
 T_3&=\{x_{13},x_{23},x_{34}\},&
 T_4&=\{x_{14},x_{24},x_{34}\}.
 \end{aligned}
\]

\begin{figure}[htbp]
\centering
\begin{tikzpicture}[x=1cm,y=1cm,
 relation/.style={draw,rectangle,fill=white,inner sep=3pt},
 number/.style={circle,fill=black,inner sep=1.6pt},
 every node/.style={font=\small}]
\node[relation] (r1) at (0,3.2) {$R_1$};
\node[relation] (r2) at (-2.7,0) {$R_2$};
\node[relation] (r3) at (2.7,0) {$R_3$};
\node[relation,dashed] (r4) at (0,1.05) {$R_4$};
\node[number,label=left:$x_{12}$] (x12) at (-1.35,1.6) {};
\node[number,label=right:$x_{13}$] (x13) at (1.35,1.6) {};
\node[number,label=below:$x_{23}$] (x23) at (0,0) {};
\node[number,label=right:$x_{14}$] (x14) at (0,2.125) {};
\node[number,label=above left:$x_{24}$] (x24) at (-1.35,0.525) {};
\node[number,label=above right:$x_{34}$] (x34) at (1.35,0.525) {};
\draw (r1)--(x12)--(r2)--(x23)--(r3)--(x13)--(r1);
\draw (r1)--(x14) (r2)--(x24) (r3)--(x34);
\draw[dashed] (r4)--(x14) (r4)--(x24) (r4)--(x34);
\end{tikzpicture}
\caption{\footnotesize The incidence graph of a Pasch configuration.
Deleting $R_4$ and its three dashed edges leaves the incidence graph
of a loose triangle with private vertices $x_{14},x_{24},x_{34}$.
}
\label{fig:zero-pasch}
\end{figure}

\begin{lemma}\label{lem:zero-three-structure}
Fix $q\in\Q_{>0}$ and a finite set $U\subseteq\R_{>0}$, and
set $B=(-1,-q,1)$. If
\[
 |X|\le2\dim(V_B(U)|_X)\qquad\text{for all $X\subseteq U$,}
\]
then every
component of the solution hypergraph $\mathcal H_q(U)$ is an
incidence tree, an incidence unicycle, or a Pasch configuration
with incidence trees attached. Each hyperedge supports exactly one
distinct vector of $\mathcal R_q(U)$.
\end{lemma}

Call a colouring \emph{strong} if every triple is rainbow.
The components can be coloured independently, and the next lemma
provides the two required constructions.

\begin{lemma}\label{lem:zero-three-colourings}
Let $\mathcal H$ be a finite three-uniform hypergraph whose components
are incidence trees, incidence unicycles, or Pasch configurations
with incidence trees attached. Then $\mathcal H$ has a strong
three-colouring. If one vertex of each triple is distinguished,
it also has a three-colouring in which the other two vertices have
different colours and the distinguished vertex has one of those colours.
\end{lemma}

Combine these two lemmas, we now finish the proof of Lemma~\ref{lem:zero-ternary-core}. 

\begin{proof}[Proof of Lemma~\ref{lem:zero-ternary-core}]
\claimsof{lem:zero-ternary-core}
Recall that $A$ is non-Schur and $\mathcal P_2(A)\ne\varnothing$,
as fixed after Proposition~\ref{prop:two-zero}.
Here $m(B)=2$, and all three coefficients of $B$ are non-zero.
We distinguish whether their sum is zero.

\begin{claim}
If the coefficients of $B$ sum to zero, the conclusion holds with
at most two colours.
\end{claim}

\begin{poc}
The core has no two-block columns-condition partition. Indeed, its
first block would have to consist of a zero-sum pair, since no
coefficient is zero; the zero total sum would then force the remaining
coefficient to be zero as well. By Lemma~\ref{lem:zero-core}, every
partition in $\mathcal P_2(A)$ therefore induces just one block on the
core.

Theorem~\ref{thm:vector-partition}, applied using
\eqref{eq:colour-evaluation-rank} and the assumed inequality
$|X|\le2\dim(V_B(U)|_X)$, partitions $U$ into two sets admitting arbitrary
prescriptions. Give these sets different colours. A proper core
solution has three points but, by~\eqref{eq:colour-solution-freedom},
its values under functions in $V_B(U)$ have dimension at most two.
Its support therefore cannot lie in either colour class. Thus no
core solution is monochromatic, excluding monochromatic $A$-solutions
and every induced one-block pattern. Finally, two colours exclude
rainbow $A$-solutions, since $k\ge3$.
\end{poc}

\begin{claim}
If the coefficients of $B$ have non-zero sum, the conclusion holds
with at most three colours.
\end{claim}

\begin{poc}
Note that every induced partition has two blocks, and
we may orient its opposite pair to write the core as $z=x+qy$,
$q\in\Q_{>0}$, with first block $\{x,z\}$.
Permuting coordinate roles and rescaling leave $V_B(U)$ unchanged.
The assumed freedom bound is therefore the hypothesis of
Lemma~\ref{lem:zero-three-structure}. This lemma describes the solution
components and makes Lemma~\ref{lem:zero-three-colourings} applicable.

If $k>3$, use the strong three-colouring supplied by that lemma.
Every core solution is rainbow, so every full $A$-solution uses at
least three colours. It is therefore neither monochromatic nor of
any prescribed two-block pattern, regardless of which opposite pair
defines the induced first block. Since the entire colouring uses
only three colours and $k>3$, no full $A$-solution is rainbow either.

It remains to consider $k=3$. The core now uses all coordinates, so
$\ker B=\ker A$ and $A$ itself has rank one. If $q=1$, the normalised
core would be a scalar and column-permutation copy of Schur's equation,
and the same would hold for $A$, contrary to the non-Schur
hypothesis on $A$.
Hence $q\ne1$. The three coefficients $-1,-q,1$ are therefore distinct,
and $\{x,z\}$ is their unique opposite pair. In particular, every
partition in $\mathcal P_2(A)$ has this first block after the same
normalisation.

Lemma~\ref{lem:zero-three-structure} also says that each unordered
solution triple supports exactly one distinct relation vector from
$\mathcal R_q(U)$. Its coefficient $-q$ thus identifies a unique
vertex $y$, which we designate as distinguished. Apply the second
colouring in Lemma~\ref{lem:zero-three-colourings} with these
distinguished vertices. Every proper solution satisfies
\[
 \chi(x)\ne\chi(z),\qquad
 \chi(y)\in\{\chi(x),\chi(z)\}.
\]
The first condition excludes monochromatic solutions and the
prescribed two-block pattern, while the second excludes rainbow
solutions. This proves the second claim.
\end{poc}

The two cases exhaust all possibilities for the coefficient sum
and prove the lemma.
\end{proof}

\subsection{Completing the non-Schur \texorpdfstring{$0$}{0}-statement}

The four colouring lemmas prove the deterministic statement;
applying extension to the random sample then gives
Proposition~\ref{prop:two-zero}. Together with
Proposition~\ref{prop:other-zero}, this proves the non-Schur
implication in Theorem~\ref{thm:zero}.

\begin{proof}[Proof of Proposition~\ref{prop:zero-colourings}]
Let $s$ and $r$ be the number of columns and rank of $B$, respectively.
If $m(B)>2$, apply Lemma~\ref{lem:zero-dense-core}. If $m(B)\le2$ and
$r\ge2$, apply Lemma~\ref{lem:zero-higher-rank}. In the remaining cases
$r=1$: Lemma~\ref{lem:zero-long-equation} applies when $s\ge4$, and
Lemma~\ref{lem:zero-ternary-core} applies when $s=3$. These cases
exhaust all density cores and give the required colouring.
\end{proof}

\begin{proof}[Proof of Proposition~\ref{prop:two-zero}]
Take $c_k=(8k)^{-1}$. Choose the density core $B$ independently of
the prescribed partition, using Lemma~\ref{lem:zero-core}. It has $s\le k$ columns and
$m(B)=m(A)=a/b$. It is regular by Lemma~\ref{lem:zero-core}.
Apply Theorem~\ref{thm:sparse-rank} over $\R$ to $S$.
Since $(8k)^{-1}\le(8s)^{-1}$, the extension property, and equivalently
the freedom bound~\eqref{eq:zero-sparse-density}, holds outside an event
of probability at most $K_sN^{-1/a}$. All relation matrices have rational
coefficients, so extending scalars from $\Q$ to $\R$ leaves their ranks
unchanged. Proposition
\ref{prop:zero-colourings} gives a single colouring avoiding
every block pattern in $\mathcal P_2(A)$, monochromatic solutions and
rainbow solutions. Take $K_k=\max_{3\le s\le k}K_s$ as required.
\end{proof}

It remains to understand Schur's equation and complete
Theorem~\ref{thm:zero}.

\medskip

\section{The Schur critical law}\label{sec:schur}

Schur's equation satisfies the extension theorem, yet two solution
triples can impose incompatible colour requirements. We trace these
conflicts through the sample and count their obstructions to determine
the limiting canonical probability at $p=cn^{-1/2}$.

Let
$\mathsf S=(1,1,-1)$ with prescribed partition
$\cI_{\mathsf S}=(\{1,3\},\{2\})$. The two possible columns-condition
partitions are equivalent by exchanging the addends.
Set $c_{\mathrm S}=\sqrt{\pi/2}$ and define
\begin{equation*}
 F(c)=
 \begin{cases}
  1-\exp(-\Lambda(c)),&0\le c<c_{\mathrm S},\\
  1,&c_{\mathrm S}\le c\le\infty,
 \end{cases}
\end{equation*}
where
\begin{equation*}
 \Lambda(c)=\frac12\bigl(\sec(c^2)+\tan(c^2)-1-c^2\bigr)
                  -\frac{c^6}{12}
 \qquad(0\le c<c_{\mathrm S}).
\end{equation*}

\begin{theorem}\label{thm:schur}
For every fixed $c\ge0$,
\[
 \Pr\bigl([n]_{cn^{-1/2}}\canarrow\mathsf S\bigr)\longrightarrow F(c).
\]
Moreover,
\[
 \sup_{0\le p\le1}
 \left|\Pr\bigl([n]_p\canarrow\mathsf S\bigr)-F(p\sqrt n)\right|
 \longrightarrow0.
\]
\end{theorem}

The function $F$ is continuous, with $0<F(c)<1$ for
$0<c<c_{\mathrm S}$. Thus no positive constant gives a $0$-statement,
while $c_{\mathrm S}$ is the least constant giving convergence to one.
These probabilities arise from the local colour constraints below.

\subsection{The minimal forcing mechanism}\label{subsec:schur-minimal}

Symmetry between the addends leaves only one canonical-free colour
pattern on each proper Schur triple. Two such constraints can
conflict on four points.

\begin{proposition}\label{prop:schur-minimal}
A colouring $\chi$ of a finite set $\Omega\subseteq\N$ is canonical-free
for $\mathsf S$ if and only if every proper triple $a+b=z$ in $\Omega$
satisfies
\begin{equation}\label{eq:schur-free-pattern}
 \chi(a)=\chi(b)\ne\chi(z).
\end{equation}
In particular, for any distinct positive integers $a,b$, the set
 $\{a,b,a+b,2a+b\}$
is canonically Schur. No set of at most three positive integers is
canonically Schur.
\end{proposition}

\begin{proof}
If the addends have different colours, the sum either makes the triple
rainbow or gives the prescribed pattern in one ordering. Otherwise,
only a different colour on the sum avoids monochromaticity. This
proves the equivalence.

Set $z=a+b$. The proper triples $a+b=z$ and $a+z=2a+b$
respectively require $\chi(a)\ne\chi(z)$ and $\chi(a)=\chi(z)$, so their
four-point support admits no canonical-free colouring. A set of at
most three points contains at most one proper Schur triple, up to
exchanging addends. Colour any such triple's addends alike and its
sum differently; if none exists, any colouring is canonical-free.
\end{proof}

Their expected number is of order $n^2p^4$, constant at
$p=cn^{-1/2}$. To obtain the full law, we must follow the same
constraints beyond four points.

\subsection{From local constraints to conflict paths}\label{subsec:schur-conflicts}

Encode the forced equalities in the \emph{equality graph}
$G(\Omega)$ on $\Omega$, joining $u\ne v$ when $u+v\in\Omega$
and labelling the edge $u+v$. Its components must be monochromatic.
By~\eqref{eq:schur-free-pattern}, a canonical-free colouring exists
exactly when no proper triple has an addend and its sum in the same
component: give different components different colours.

A conflicting equality path closes with an inequality. When all
involved points are distinct, this gives the following forcing core.

For $t\ge2$, choose positive integers $u_1,\ldots,u_t$ with $u_1<u_2$ and set
\begin{equation}\label{eq:schur-forcing}
 w_1=u_2-u_1,\qquad
 w_i=u_i+u_{i+1}\quad(2\le i\le t),\qquad u_{t+1}=u_1.
\end{equation}
If all $2t$ numbers $u_i,w_i$ are distinct, their set is a
\emph{forcing cycle core of length $t$}. Every set containing such a core
is canonically Schur: the triples indexed by $2,\ldots,t$ force
$\chi(u_2)=\cdots=\chi(u_t)=\chi(u_1)$, whereas the first triple forces
$\chi(u_1)\ne\chi(u_2)$.

The four-point mechanism has $t=2$, with $u_1=a$ and $u_2=a+b$.
Shortest conflicts separate the vertices, labels and closing
difference; only labels may repeat.

An \emph{admissible path} of length $\ell\ge0$ is an ordered list
$u_0,\ldots,u_\ell$ of distinct elements of $[n]$ such that its labels
$w_i=u_{i-1}+u_i$ lie in $[n]$ and
\[
 \{u_0,\ldots,u_\ell\}\cap\{w_1,\ldots,w_\ell\}=\varnothing.
\]
The labels may repeat. Its support is the union of its vertices and
labels, and the path is \emph{present} if its support lies in $\Omega$.

\begin{lemma}\label{lem:schur-shortest}
Let $n\ge1$ and $\Omega\subseteq[n]$. If $\Omega$ is canonically Schur,
there is a present admissible path $u_0,\ldots,u_\ell$, where $\ell\ge1$
and $u_0<u_\ell$, with labels $w_i=u_{i-1}+u_i$ and closing difference
\[
 b=u_\ell-u_0\in\Omega\setminus
 \bigl(\{u_0,\ldots,u_\ell\}\cup\{w_1,\ldots,w_\ell\}\bigr).
\]
If its labels are distinct, its support together with $b$ is a forcing
cycle core of length $\ell+1$.
\end{lemma}

\begin{proof}
Choose a proper triple $a+b=z$ with minimum finite addend-to-sum distance
in $G(\Omega)$, minimizing over both addends in every proper triple. Let
$u_0=a,u_1,\ldots,u_\ell=z$ be a shortest path. It is simple, and
$u_0<u_\ell$ by positivity.

If $w_i=u_j$, the triple on the edge $u_{i-1}u_i$ has an addend at distance
at most
\[
 \min\{|j-i+1|,|j-i|\}\le\ell-1
\]
from its sum along the chosen path, a contradiction. The path is
therefore admissible. Also $b$ cannot be an internal path vertex, since
its distance to $z$ would then be less than $\ell$; properness and
positivity exclude the endpoints.

Suppose $b=w_i$. The original triple gives the equality
edge $u_0b$. If $i<\ell$, this edge followed by the path from $u_0$ to
$u_{i-1}$ gives distance at most $i<\ell$ from the sum $b=w_i$ to an
addend of its triple. If $i=\ell$, then
$w_\ell=u_{\ell-1}+u_\ell>u_\ell>b$. Both are impossible, so $b$ is distinct from every path vertex and label.
When the labels are distinct, traverse the path backwards from
$u_\ell$ to $u_0$ after the first difference edge. The resulting
configuration has the form~\eqref{eq:schur-forcing}, with $2\ell+2$
distinct points.
\end{proof}

\subsection{Weighted path counting}\label{subsec:schur-paths}
To reduce shortest conflicts to bounded cores, we control
repeated labels and long paths. A new label needs two sampled
points; reusing a label needs only one but determines the endpoint.
Weighted counting exploits this trade-off.

\begin{proposition}\label{prop:schur-reduction}
For every fixed $0<c<c_{\mathrm S}$, set $p=cn^{-1/2}$. As $n\to\infty$,
\[
 \Pr\bigl([n]_p\text{ is canonically Schur but contains no
 forcing cycle core}\bigr)=O_c(n^{-1/2}).
\]
For every $\tau\in(c^2,\pi/2)$, set
 $q=\frac{c^2}{\tau}$ and $H=\frac{1+\sin\tau}{\cos\tau}.$
Then, for $T\ge2$ and $n$ with $p\le1$,
\[
 \Pr\bigl([n]_p\text{ contains a forcing cycle core of length
 greater than }T\bigr)
 \le\frac{c^2Hq^T}{1-q}.
\]
\end{proposition}

\begin{proof}
Fix $\tau\in(c^2,\pi/2)$ and set
$h(x)=H\cos(\tau x)-\sin(\tau x)$. This decreasing positive weight
contracts the sum over possible next vertices: it satisfies
\[
 1\le h(x)\le H,\qquad
 h(x)=1+\tau\int_0^{1-x}h(y)\,dy.
\]
Its right Riemann sums therefore give, for every $u\in[n]$,
\begin{equation}\label{eq:schur-contraction}
 p^2\sum_{v=1}^{n-u}h(v/n)
 \le c^2\int_0^{1-u/n}h(y)\,dy
 =q\bigl(h(u/n)-1\bigr)\le qh(u/n).
\end{equation}
For the first assertion take $\tau=(c^2+\pi/2)/2$; otherwise keep
$\tau$ arbitrary.

Give an admissible path $P=(u_0,\ldots,u_\ell)$ weight
$p^{|\operatorname{supp}P|}h(u_\ell/n)$. Let $W_\ell$ be the sum of these
weights, with $U_\ell$ and $R_\ell$ the parts having distinct and
repeated labels. Since $h\ge1$, these sums bound the corresponding
expected path counts. For $\ell\ge1$, we claim
\[
 \begin{aligned}
 U_\ell&\le npHq^\ell,\qquad
 W_\ell&\le npH(q+\ell pH)^\ell,\qquad
 R_\ell&\le np^2H^2\ell^2(q+\ell pH)^{\ell-1}.
 \end{aligned}
\]
The initial weight is $W_0\le npH$. An admissible extension adds a new
path vertex. If its label is also new, it adds two support points, and
\eqref{eq:schur-contraction} bounds its total weighted contribution by
$q$ times the previous weight. If its label repeats, choosing a previous
label determines the new vertex. During a path of final length $\ell$
there are at most $\ell$ such choices, each adding one support point;
since $1\le h\le H$, their contribution is at most $\ell pH$ times the
previous weight. This proves the first two bounds. For the third, mark
the first repeated-label step. There are at most $\ell$ possible
positions. The marked step contributes a factor of at most $\ell pH$,
and each remaining step contributes a factor of at most $q+\ell pH$.
Thus
\[
 R_\ell\le npH\cdot\ell(\ell pH)(q+\ell pH)^{\ell-1},
\]
which proves the third bound.

Set $q'=(1+q)/2$ and
$L=\lceil2\log n/|\log q'|\rceil$. For sufficiently large $n$,
$LpH\le q'-q$, and hence
\[
 \Pr(\text{a present admissible path of length at least }L)
 \le W_L\le npH(q')^L=O_c(n^{-3/2}).
\]
Every longer admissible path has an admissible prefix of length $L$.

By Lemma~\ref{lem:schur-shortest}, a canonically Schur set without a forcing
cycle core has a shortest conflict with a repeated label. If its length
is less than $L$, its closing difference is a new point. Requiring it
to be retained contributes a factor $p$, so the probability is at
most
\[
 p\sum_{\ell=1}^{L-1}R_\ell
 \le np^3H^2\sum_{\ell\ge1}\ell^2(q')^{\ell-1}
 =O_c(n^{-1/2}).
\]
Together with the long-path bound, this proves the first assertion.

A forcing cycle core of length $t>T$ contains an admissible path of
length $\ell=t-1\ge T$ with distinct labels and an additional closing
difference. Therefore its probability is at most
\begin{align*}
p\sum_{\ell\ge T}U_\ell
 \le np^2H\sum_{\ell\ge T}q^\ell
 =\frac{c^2Hq^T}{1-q}.\tag*{\qedhere}    
\end{align*}
\end{proof}

\subsection{Counting forcing cores}\label{subsec:schur-counts}

The tail bound leaves bounded-length cores to count. Alternating
permutations count their parameter tuples; passing to supports
changes the leading coefficient only at length three, through a
Pasch symmetry.

For integers $t\ge2$ and $n\ge1$, let $\mathcal C_{t,n}$ be the family
of distinct forcing-core supports of length $t$ in $[n]$. Write $E_t$
for the number of permutations $\sigma_1\cdots\sigma_t$ of $[t]$
satisfying $\sigma_1<\sigma_2>\sigma_3<\cdots$, with $E_0=1$.
Set $a_t=E_t/(2t!)$ for $t\ge2$, $t\ne3$, and $a_3=1/12$.
For each $t\ge2$, let $e_1,\ldots,e_t$ be the standard basis of $\Q^t$.
The coordinate forms of a forcing core are
\[
 \mathcal V_t=\{e_1,\ldots,e_t,D,Q_2,\ldots,Q_t\},
 \qquad
 D=e_2-e_1,\quad Q_i=e_i+e_{i+1},\quad e_{t+1}=e_1.
\]

\begin{lemma}\label{lem:schur-core-counts}
For every fixed integer $t\ge2$, as $n\to\infty$,
\[
 |\mathcal C_{t,n}|=a_tn^t+O_t(n^{t-1}).
\]
In particular, $a_2=1/4$ and $a_3=1/12$.
\end{lemma}

\begin{proof}
We first count ordered tuples $(u_1,\ldots,u_t)$ defining forcing cycle cores.
\claimsof{lem:schur-core-counts}
\begin{claim}\label{cl:schur-parameter-count}
The number of ordered $(u_1,\ldots,u_t)$ defining forcing
cycle cores of length $t$ in $[n]$ is
\[
 \frac{E_t}{2t!}n^t+O_t(n^{t-1}).
\]
\end{claim}

\begin{poc}
Set $(x_1,\ldots,x_t)=(u_2,\ldots,u_t,u_1)/n$.
Before imposing $u_1<u_2$ and distinctness of the $2t$ points,
the parameter domain, up to its boundary, is
\[
 \mathcal P_t=
 \{x\in[0,1]^t:x_i+x_{i+1}\le1\ (1\le i<t)\}.
\]
The substitution $y_i=x_i$ for odd $i$ and $y_i=1-x_i$ for even
$i$ preserves volume and changes these inequalities into
$y_1\le y_2\ge y_3\le\cdots$. Of the $t!$ possible relative
orders of the coordinates, exactly $E_t$ satisfy this pattern.
Each order occupies volume $1/t!$, so
$\operatorname{vol}(\mathcal P_t)=E_t/t!$.

Reversing the coordinates preserves $\mathcal P_t$ and exchanges
the regions $x_t<x_1$ and $x_1<x_t$. Their common boundary has
volume zero. Thus imposing $u_1<u_2$, equivalently $x_t<x_1$,
halves the volume. The difference $u_2-u_1$ then automatically
lies in $[n]$. Lattice-point counting in this fixed polytope gives
\[
 \frac{E_t}{2t!}n^t+O_t(n^{t-1})
\]
integer parameter tuples with all coordinates positive and all
the required sums at most $n$.

The $2t$ coordinate forms in $\mathcal V_t$ are pairwise distinct.
Consequently, an equality between two of their values restricts
the parameters to a proper rational hyperplane. There are only
$O_t(1)$ such hyperplanes, each containing $O_t(n^{t-1})$ integer
tuples in $[n]^t$. Removing these tuples imposes the required
distinctness without changing the leading term.
\end{poc}

\begin{claim}\label{cl:schur-support-multiplicity}
Apart from $O_t(n^{t-1})$ parameter tuples, every support has
exactly one parameterization when $t\ne3$, and exactly two when
$t=3$.
\end{claim}

\begin{poc}
We first identify the Schur triples that hold identically among
the coordinate forms. The unit vectors, $D$, and the $Q_i$ have
coefficient sums $1$, $0$, and $2$, respectively. An identity
between three distinct coordinate forms must therefore be of one
of the following types
\[
 e_i+e_j=Q_r,\qquad D+e_i=e_j,\qquad D+Q_i=Q_j.
\]
The first two give precisely the displayed cycle triples in
\eqref{eq:schur-forcing}. In the third, the negative $e_1$
coefficient of $D$ must be cancelled by $Q_i$, forcing $i=t$.
But $D+Q_t=e_2+e_t$ is a listed form only when $t=3$, when it
equals $Q_2$.

Any additional numerical Schur triple thus imposes a non-zero
linear equation on the parameters. Exclude the tuples on these
finitely many proper rational hyperplanes; there are
$O_t(n^{t-1})$ of them. Call the remaining tuples generic. Their
supports contain exactly the cycle triples, together with the
one additional triple just identified when $t=3$.

Suppose first that $t\ne3$, and consider the hypergraph of all
proper Schur triples on a generic support. Its degree-two
vertices are precisely $u_1,\ldots,u_t$; each $w_i$ has degree
one. The unique triple whose sum is a degree-two vertex is
$u_1+w_1=u_2$. It recovers $u_2$ as that sum and $u_1$ as the
other degree-two vertex of the triple. Deleting this triple
leaves a path from $u_2$ to $u_1$ through the degree-two
vertices. Following it recovers $u_3,\ldots,u_t$ in order,
and its remaining vertices recover $w_2,\ldots,w_t$.
This also applies when $t=2$, when the path consists of a single
triple. Thus the parameterization is unique.

For $t=3$, write $a=u_2-u_1$, $b=u_1$ and $c=u_3$. The support
is
\[
 \{a,b,c,a+b,b+c,a+b+c\},\qquad a,b,c>0,
\]
and its four Schur triples form a Pasch configuration. The
largest point is $a+b+c$. The two triples not containing it
are $\{a,b,a+b\}$ and $\{b,c,b+c\}$, whose common point is $b$.
Since the sum in each triple is determined by its numerical
values, these two triples recover $b$ and the unordered pair
$\{a,c\}$. Hence the only possible parameterizations are
\[
 (u_1,u_2,u_3)=(b,a+b,c)
 \quad\text{and}\quad
 (u_1,u_2,u_3)=(b,b+c,a).
\]
Both satisfy the defining conditions and give the same support.
They are distinct because the support has six distinct points,
so $a\ne c$.

Every parameterization uses the same Schur triples, so this recovery
also accounts for any obtained from excluded tuples.
\end{poc}

Every excluded tuple contributes at most one support. Claims~\ref{cl:schur-parameter-count}
and~\ref{cl:schur-support-multiplicity} therefore give
\[
 |\mathcal C_{t,n}|=
 \begin{cases}
  \dfrac{E_t}{2t!}n^t+O_t(n^{t-1}),&t\ne3,\\[1ex]
  \dfrac{E_3}{4\cdot3!}n^3+O(n^2),&t=3.
 \end{cases}
\]
Since $E_2=1$ and $E_3=2$, this gives the stated coefficients,
in particular $a_2=1/4$ and $a_3=1/12$.
\end{proof}

At $p=cn^{-1/2}$, the expected number of length-$t$ cores tends to
$a_tc^{2t}$. Andr\'e's identity for alternating
permutations~\cite[Theorem~1.1]{Stanley} gives
\[
 \sum_{t\ge0}E_t\frac{z^t}{t!}=\sec z+\tan z
 \qquad(0\le z<\pi/2).
\]
Hence, for $0\le c<c_{\mathrm S}$,
\begin{equation}\label{eq:schur-core-series}
 \sum_{t\ge2}a_tc^{2t}
 =\frac12\bigl(\sec(c^2)+\tan(c^2)-1-c^2\bigr)-\frac{c^6}{12}
 =\Lambda(c).
\end{equation}
The first two terms, $c^4/4$ and $c^6/12$, count the four-point
mechanism and Pasch supports. To turn this mean into a limiting
probability, we control overlaps.

\subsection{Avoidance probabilities and the critical law}\label{subsec:schur-limit}

\begin{lemma}\label{lem:schur-core-avoidance}
Fix an integer $L\ge3$ and a real number $c>0$, and set $p=cn^{-1/2}$.
Let $Y_L$ count the members of
$\bigcup_{t=2}^L\mathcal C_{t,n}$ contained in $[n]_p$. Then, as $n\to\infty$,
\[
 \Pr(Y_L=0)\longrightarrow e^{-\Lambda_L(c)},
 \qquad
 \Lambda_L(c)=\sum_{t=2}^L a_tc^{2t}.
\]
\end{lemma}

A full core has $2t$ points and $t$ parameters. Proper subsets of its
coordinate forms $\mathcal V_t$ satisfy a strict rank inequality,
which makes overlapping pairs negligible at $p=cn^{-1/2}$.

\begin{lemma}\label{lem:schur-core-rank}
For $t\ge2$, every nonempty proper subset $X\subsetneq\mathcal V_t$
satisfies $\rank(X)>|X|/2$.
\end{lemma}

\begin{proof}
For $t=2$, every pair of coordinate forms is independent, so the claim
is immediate. For $t\ge3$, associate $D=e_2-e_1$ with the edge $12$
and $Q_i=e_i+e_{i+1}$ with the edge $i(i+1)$, with indices modulo $t$.
These vectors are supported on the edges of a cycle on the coordinate
directions. If not all these edges are
selected, their non-trivial components are paths.  A path on $v$ directions
has $v-1$ edge vectors, of rank $v-1$, as successive elimination of
leaf coordinates shows. With no selected unit vector it
satisfies the required strict inequality; with at least one selected unit
vector its rank is $v$, while it contains at most $2v-1$ selected vectors.
Isolated selected units also satisfy the inequality, and ranks add over
these components.

If all cycle edges are selected, their rank is at least $t-1>t/2$.
Adding any unit vector gives rank $t$, since a spanning path together
with that unit already spans all $t$ directions.  As $X$ is proper,
it then has at most $2t-1$ elements.
\end{proof}

\begin{proof}[Proof of Lemma~\ref{lem:schur-core-avoidance}]
Lemma~\ref{lem:schur-core-counts} gives
$\mu_L:=\E Y_L\to\Lambda_L(c)$.  Write
$\mathcal W=\bigcup_{t=2}^L\mathcal C_{t,n}$ and define 
\[
 \Delta_L=
 \sum_{\substack{F,G\in\mathcal W\\F\ne G,\ F\cap G\ne\varnothing}}
 p^{|F\cup G|}.
\]
We claim that $\Delta_L=O_{L,c}(n^{-1/2})$.
Indeed, suppose $|F|=2t$, $|G|=2s$, and $|F\cap G|=h$, and fix the
coordinate roles of the common points.  Identifying these forms imposes
linear equations on the $t+s$ parameters, of rank $R$ at least the
rank of the common forms in either core.  The common set is proper in
at least one core, since $F\ne G$, so Lemma~\ref{lem:schur-core-rank} gives
$R>h/2$.  There are $O_L(n^{t+s-R})$ parameter pairs with these roles,
and their contribution is
\[
 O_{L,c}(n^{t+s-R}p^{2t+2s-h})
 =O_{L,c}(n^{h/2-R})=O_{L,c}(n^{-1/2}).
\]
There are only finitely many role patterns for fixed $L$.
Harris' and Janson's inequalities~\cite{JLR} give
\[
 \prod_{F\in\mathcal W}(1-p^{|F|})
 \le\Pr(Y_L=0)
 \le\exp(-\mu_L+\Delta_L/2).
\]
Since $\sum_{F\in\mathcal W}p^{2|F|}\le p^4\mu_L=o(1)$, both bounds
tend to $e^{-\Lambda_L(c)}$.
\end{proof}

\begin{proof}[Proof of Theorem~\ref{thm:schur}]
Write $P_n(p)$ for the probability of the canonical Schur property.
First fix $0<c<c_{\mathrm S}$, and let $L\ge3$.
By~\eqref{eq:schur-core-series}, $\Lambda_L(c)\to\Lambda(c)$ as
$L\to\infty$.
Every forcing core is canonically Schur.  Conversely,
Proposition~\ref{prop:schur-reduction} bounds the probability of a canonically Schur
set without a forcing core by $O_c(n^{-1/2})$, and gives a geometric
tail for the lengths of the cores.  Thus, for
$\varepsilon_L(c)\to0$ as $L\to\infty$,
\[
 0\le P_n(cn^{-1/2})-\Pr(Y_L>0)
 \le O_c(n^{-1/2})+\varepsilon_L(c).
\]
Apply Lemma~\ref{lem:schur-core-avoidance}, first let $n\to\infty$, and then
let $L\to\infty$.  We obtain
$P_n(cn^{-1/2})\to1-e^{-\Lambda(c)}=F(c)$.

All $a_t$ are positive, so $0<F(c)<1$ for $0<c<c_{\mathrm S}$.
The explicit expression gives $\Lambda(c)\to\infty$ as $c\uparrow c_{\mathrm S}$,
and hence $F(c)\to1$.  For $c\ge c_{\mathrm S}$, monotonicity gives
\[
 \liminf_{n\to\infty}P_n(cn^{-1/2})\ge F(c_0)
 \qquad(0<c_0<c_{\mathrm S}).
\]
Letting $c_0\uparrow c_{\mathrm S}$ proves convergence to one.  The case $c=0$
is immediate.  Thus $F$ is continuous on $[0,\infty]$, with $F(\infty)=1$, and
$c_{\mathrm S}$ is the least constant giving convergence to one.

Finally, consider any sequence $p_n\in[0,1]$ for which
$p_n\sqrt n\to c\in[0,\infty]$.  If $c<\infty$, comparison with
$(c-\varepsilon)_+n^{-1/2}$ and $(c+\varepsilon)n^{-1/2}$, followed by
$\varepsilon\downarrow0$, gives $P_n(p_n)\to F(c)$.
If $c=\infty$, comparison with $c_{\mathrm S}n^{-1/2}$ gives the same conclusion.
Were the asserted uniform approximation false, a violating sequence
$p_n$ would have a subsequence for which $p_n\sqrt n$ converges in
$[0,\infty]$.  The preceding conclusion and continuity of $F$ would
then give $|P_n(p_n)-F(p_n\sqrt n)|\to0$, a contradiction.
\end{proof}

\begin{proof}[Proof of Theorem~\ref{thm:zero}]
Propositions~\ref{prop:other-zero} and~\ref{prop:two-zero},
with $S=[n]$, prove \textup{(i)}$\Rightarrow$\textup{(ii)}.
For Schur's equation the two columns-condition partitions are
equivalent by exchanging the addends. Row rescaling and column
permutation preserve the canonical property, and $m(A)=2$;
Theorem~\ref{thm:schur} therefore gives a positive limiting
probability at $p=cn^{-1/2}$ for every $c>0$, ruling out
\textup{(ii)} for every scalar and column-permutation copy of
Schur's equation.
\end{proof}

\section{The 1-statement}\label{sec:counting}

Throughout this section, fix an irredundant partition-regular
matrix $A\in\Z^{\ell\times k}$ with $k\ge3$ and
$\rank A=\ell\ge1$, and an ordered partition $\cI$ witnessing
the columns condition. Write $d=k-\ell$ and $m=m(A)$.
All canonical patterns refer to $\cI$, and constants may depend
on $A,\cI$ and the other fixed parameters.

\subsection{Proof strategy and notation}\label{subsec:one-notation}

We must find a canonical solution in every colouring of the random
sample, regardless of the number of colours. We will reduce to a
fixed palette by keeping large original colour classes unchanged
and merging small ones. The unchanged classes receive
\emph{protected} labels, and the merged classes receive
\emph{light} labels. In a palette $[R]$, we protect $[H]$ and
regard the remaining labels as light; unused labels are allowed.
Merging can create the equalities inside a prescribed block.
We therefore seek solutions whose repeated labels all come from
unchanged classes.

\begin{definition}[Admissible solutions]\label{def:admissible}
Let $\chi$ be a colouring of $U\subseteq\N$, with some labels
designated as protected. A proper solution
$\mathbf x=(x_1,\ldots,x_k)\in U^k$ of $A\mathbf x=0$ is
\emph{admissible} if it is canonical and every colour appearing
on two or more coordinates is protected.
\end{definition}

Every rainbow solution is admissible. If protected classes are
unchanged original classes and no original class is split,
the entire equality pattern of an admissible solution survives
undoing the mergers. This is the link between the following
three parts of the proof.

Subsection~\ref{subsec:colour-reduction} proves
Lemma~\ref{lem:reduction} for any prescribed $a\in(0,1/4]$.
It produces a fixed palette $[R]$, with $[H]$ protected, no
deleted vertices, and light classes of size at most $a/4$
times the sample size. Admissible solutions retain their original
colour pattern. We leave $a$ open until the dense count.

Subsection~\ref{subsec:supersaturation} supplies this count in
Lemma~\ref{lem:densecount}: every colouring of a set of density
sufficiently close to one has $\beta n^d$ admissible solutions
if each light class has size at most $an$. Its constants are
independent of $H,R$, so we can now fix $a$ and the resulting
palette. We give light vertices distinct fresh colours, apply
a uniform canonical count, and bound the losses when their
labels are restored. Appendix~\ref{app:finite} proves the
canonical count by including a basis of $\ker A$ in the Deuber
construction to retain all the required independent parameters.

Finally, Subsection~\ref{subsec:bounded-palette} proves
Proposition~\ref{prop:bounded-palette}: above $Cn^{-1/m}$,
with exponentially high probability every colouring of the sample
with this palette and light-class bound has an admissible
solution. We encode solutions as labelled hyperedges and cover
colourings without one by containers. A capacitated matching
extracts a fixed set on which a supported colouring would force
a density excess. Conditioning on small fingerprints and taking
a union bound excludes these excesses. Colour reduction and
restoration of the original colours then prove
Theorem~\ref{thm:one}.

\subsection{Colour reduction}\label{subsec:colour-reduction}

We first reduce to a bounded palette while preserving the
colour patterns we need. The light-class bound is a parameter;
the next subsection will provide a suitable value.

\begin{lemma}\label{lem:reduction}
For every $a\in(0,1/4]$, there are integers $1\le H<R$ such
that every colouring $\chi$ of a finite $W\subseteq\N$ admits
a colouring $\varphi:W\to[R]$, with $[H]$ protected, satisfying:
\begin{enumerate}
\item each $\varphi$-colour class is a union of whole
$\chi$-colour classes, and each protected class is either
empty or one $\chi$-colour class;
\item $|\varphi^{-1}(i)|\le a|W|/4$ for every $H<i\le R$;
\item every proper admissible $A$-solution for $\varphi$ is
canonical for $\chi$.
\end{enumerate}
No vertices are deleted, and unused labels are allowed.
\end{lemma}

\begin{proof}
Take $H=\lceil8/a\rceil$ and $R=2H+1$. For nonempty $W$,
give distinct protected labels to all original classes larger
than $a|W|/8$; there are fewer than $8/a$ of them. Greedily
group the remaining classes into unions of size between
$a|W|/8$ and $a|W|/4$, apart from a possible smaller final
union. There are at most $H+1=R-H$ unions, which receive
distinct light labels. Unused labels are left empty.
This proves \textup{(i)} and \textup{(ii)}. In an admissible
solution, every repeated label comes from an unchanged original
class. Since no class was split, the entire equality pattern
survives undoing the mergers, proving \textup{(iii)}.
For $W=\varnothing$, use the empty colouring.
\end{proof}

To choose $a$, we next count admissible solutions in dense
sets, uniformly over the protected and light palettes.

\subsection{Uniform supersaturation}\label{subsec:supersaturation}

The reduction leaves us free to choose the light-class bound.
The following estimate supplies one that forces many admissible
solutions in dense sets, independently of the palette.

\begin{lemma}\label{lem:densecount}
There are $\sigma\in(0,1)$, $a\in(0,1/4]$ and $\beta>0$ such
that, for all sufficiently large $n$ and all integers
$0\le H\le R$ with $R\ge1$, every colouring
$\chi:S\to[R]$ of $S\subseteq[n]$ with
\[
 |S|\ge(1-\sigma)n,\qquad
 |\chi^{-1}(i)|\le an\quad(H<i\le R)
\]
has at least $\beta n^d$ ordered proper admissible solutions,
where $[H]$ are protected.
\end{lemma}

We prove this by counting canonical solutions after giving each
light vertex a distinct colour. The following consequence of
Lefmann's theorem supplies the count before any palette is fixed.

\begin{lemma}\label{lem:canonical-supersaturation}
There are $\sigma\in(0,1)$ and $\beta_0>0$ such that, for all
sufficiently large $n$, every colouring of every $S\subseteq[n]$
with $|S|\ge(1-\sigma)n$ has at least $\beta_0n^d$ ordered
proper canonical solutions.
\end{lemma}

Its proof is given in Appendix~\ref{app:finite}. To control the
loss when light labels are restored, we need a bound on the
number of solutions through two prescribed coordinates. The
same calculation gives all the degree estimates needed later
for containers.

For a hypergraph $H$, write $V(H)$ for its vertex set, $v(H)$
for its number of vertices and $e(H)$ for its number of edges.
Let $\Delta_j(H)$ be the largest number of edges containing a
fixed set of $j$ distinct vertices, and let $H[U]$ denote the
induced hypergraph on $U$. Ordered edges are tuples of distinct
vertices; all edge counts and degrees include multiplicities.

\begin{lemma}\label{lem:degrees}
Let $H_n$ be the ordered hypergraph on $[n]$ whose edges are
the proper solutions of $A\mathbf x=0$. There is $K_A>0$,
depending only on $A$, such that, for every integer $n\ge1$,
\[
 e(H_n)\le n^d,\qquad
 \Delta_j(H_n)\le K_A n^{d-1-(j-1)/m}
 \quad(1\le j\le k).
\]
Prescribing values in any two distinct coordinate positions
leaves at most $n^{d-2}$ solutions in $[n]^k$.
\end{lemma}

\begin{proof}
For $J\subseteq[k]$, equation~\eqref{eq:projection-rank} gives
\[
 \dim\pi_J\ker A=|J|-\ell+\rank A_{\overline J}.
\]
Prescribing the coordinates in $J$ leaves an affine fibre of
dimension $d-\dim\pi_J\ker A$, if it is nonempty. Projecting
injectively onto that many coordinate positions bounds its
number of points in $[n]^k$ by $n^{d-\dim\pi_J\ker A}$.
For $|J|\ge2$, the definition of $m$ gives
$\dim\pi_J\ker A\ge1+(|J|-1)/m$. For a singleton this
dimension is one, by Lemma~\ref{lem:rank}. There are at most
$k!$ ways to place a prescribed set of values in coordinate
positions, proving the degree bound. The same projection
argument gives the total count $n^d$.
Finally, Lemma~\ref{lem:rank} gives
$\dim\pi_{\{i,j\}}\ker A=2$ for $i\ne j$, which proves the
two-coordinate completion bound.
\end{proof}

We now use these bounds to prove the admissible count.

\begin{proof}[Proof of Lemma~\ref{lem:densecount}]
Take $\sigma,\beta_0$ from Lemma~\ref{lem:canonical-supersaturation} and set
\[
 a=\min\left\{\frac14,\frac{\beta_0}{2\binom{k}{2}}\right\},
 \qquad \beta=\frac{\beta_0}{2}.
\]
Give every light vertex a distinct fresh colour, different from
the protected labels. Lemma~\ref{lem:canonical-supersaturation}
gives at least $\beta_0n^d$ proper canonical solutions in this
colouring. Restoring the light labels changes a counted equality
pattern only if two coordinates fall in the same light class.
For a fixed pair of positions, there are at most
\[
 \sum_{\text{light }C}|C|^2\le an\sum_C|C|\le an^2
\]
choices for their values, each with at most $n^{d-2}$ completions
by Lemma~\ref{lem:degrees}. At most $\binom{k}{2}an^d$ counted solutions are lost. Every
remaining solution has its original canonical pattern and no
repeated light label, and hence is admissible.
\end{proof}

The constant $a$ now determines $H,R$ through the reduction.
It remains to transfer the dense conclusion to every colouring
of the random sample with this fixed palette.

\subsection{Containers and the random sample}\label{subsec:bounded-palette}

In this subsection, fix integers $1\le H\le R$, protect $[H]$,
and fix $\sigma\in(0,1)$, $a\in(0,1/4]$ and $\beta>0$ with
the following dense property: for all sufficiently large $n$,
every colouring $\chi:S\to[R]$ with $S\subseteq[n]$,
$|S|\ge(1-\sigma)n$ and
$|\chi^{-1}(i)|\le an$ for $H<i\le R$ has at least
$\beta n^d$ ordered proper admissible solutions.
Lemma~\ref{lem:densecount} supplies these constants. Our goal
is to transfer this property to the random sample.

\begin{proposition}\label{prop:bounded-palette}
There are $C,c>0$ such that, for all sufficiently large $n$ and
$Cn^{-1/m}\le p\le1$, with probability at least $1-e^{-cnp}$,
every colouring $\varphi:\Omega\to[R]$ of $\Omega=[n]_p$
satisfying
\[
 |\varphi^{-1}(i)|\le\frac a4|\Omega|
 \qquad\text{for all $H<i\le R$}
\]
contains a proper admissible $A$-solution.
\end{proposition}

We prove the proposition by covering colourings without an
admissible solution by containers. We first encode the colour
choices in a hypergraph, then extract a fixed density obstruction
from each container, and finally bound the probability that any
small fingerprint witnesses such an obstruction.

Define the $k$-uniform multi-hypergraph
$G_n=G_n(A,\cI;H,R)$ as follows.
\begin{itemize}[leftmargin=*,itemsep=0.25em]
\item \textbf{Vertex set:} $V(G_n)=[n]\times[R]$.
The vertex $(x,i)$ assigns the integer $x$ label $i$, and
$v(G_n)=Rn$.
\item \textbf{Hyperedges:} for each ordered proper solution
$\mathbf x=(x_1,\ldots,x_k)\in[n]^k$ and each admissible canonical
label list $(i_1,\ldots,i_k)\in[R]^k$, add one copy of
\[
 \{(x_1,i_1),\ldots,(x_k,i_k)\}.
\]
Thus every repeated label in the list belongs to $[H]$.
\end{itemize}
Each edge uses $k$ distinct integers. We retain one copy for each
ordered solution and label list, even when different pairs give
the same edge. All edge counts and degrees include multiplicities.
The hypergraph is fixed before any colouring or random sampling.

A colouring $\varphi:U\to[R]$ of $U\subseteq[n]$ selects the set
\[
 X_\varphi=\{(x,\varphi(x)):x\in U\}.
\]
The edges of $G_n[X_\varphi]$, with multiplicities, correspond
exactly to the ordered admissible solutions in $\varphi$.
In particular, a colouring with no admissible solution gives an
independent set. We say that $\mathcal C\subseteq[n]\times[R]$
\emph{supports} $\varphi$ if $X_\varphi\subseteq\mathcal C$.

The following theorem of Balogh, Morris and
Samotij~\cite[Theorem~2.2]{BMS} covers independent sets by
containers. It permits repeated edges, counting edges and degrees
with multiplicities.

\begin{theorem}[Balogh--Morris--Samotij]\label{thm:containers}
Fix $k\ge2$, $K>0$ and $0<\varepsilon_0<1$. There is $B>0$
with the following property. Let $H$ be a $k$-uniform
multi-hypergraph, and let $\mathcal F$ be an increasing family of
subsets of $V(H)$ such that every $U\in\mathcal F$ satisfies
\[
 |U|>\varepsilon_0 v(H),\qquad
 e(H[U])>\varepsilon_0 e(H).
\]
If $0<\tau<1$ and
\[
 \Delta_j(H)\le K\tau^{j-1}\frac{e(H)}{v(H)}
 \qquad(1\le j\le k),
\]
then there is a family $\mathcal S$ of subsets of $V(H)$,
each of size at most $B\tau v(H)$, and a map
$f:\mathcal S\to\mathcal P(V(H))\setminus\mathcal F$ such that every
independent set $I$ in $H$ has some $S\in\mathcal S$ with
$S\subseteq I$ and $I\setminus S\subseteq f(S)$.
\end{theorem}

A fingerprint $S$ determines the container $f(S)\cup S$.
To use the dense property, we must understand which colourings
a container can support. The next lemma turns a shortage of
admissible solutions into a density excess on one fixed set;
the capacities on light labels encode their class-size bounds.

\begin{lemma}\label{lem:container-certificate}
For all sufficiently large $n$, every
$\mathcal C\subseteq[n]\times[R]$ satisfying
$e(G_n[\mathcal C])<\beta n^d$ determines a set $T\subseteq[n]$
and a number $t\in(\sigma,1]$ such that
\[
 |T|<(t-\sigma)n.
\]
If $\mathcal C$ supports a colouring $\varphi:W\to[R]$ of
$W\subseteq[n]$ with every light class of size at most
$a|W|/4$, then
\[
 |W\cap T|\ge t|W|.
\]
The pair $(T,t)$ depends only on $\mathcal C$ and the fixed
parameters, and can be chosen before sampling $W$.
\end{lemma}

\begin{proof}
Set $b=\lceil an/4\rceil\le an$, taking $n$ sufficiently large.
Consider the bipartite graph joining $x\in[n]$ to each label
$i\in[R]$ allowed by $\mathcal C$. Give each integer capacity
one, each protected label capacity $n$, and each light label
capacity $b$. Let $M$ be the maximum number of matched integers.
If $M\ge(1-\sigma)n$, the matching gives a supported colouring
of at least $(1-\sigma)n$ points with light classes of size at
most $an$. The dense property fixed above then gives at least
$\beta n^d$ admissible solutions in $\mathcal C$, contrary to
the edge bound. Hence $M<(1-\sigma)n<n$.

Fix a maximum matching and start alternating paths at all
unmatched integers. Every reachable label is saturated, since
otherwise an augmenting path exists. No protected label is
reachable, as its capacity $n$ exceeds $M$. Let $J$ be the set
of reachable labels and $T$ the set of unreachable integers.
Every integer in $T$ is matched, and exactly $b|J|$ matched
integers are reachable. Moreover, every allowed label of a
reachable integer is reachable. Thus
\[
 |T|+b|J|=M,\qquad
 x\notin T\ \Longrightarrow
 \{i:(x,i)\in\mathcal C\}\subseteq J.
\]

Put $q=b|J|/n$ and $t=1-q$. Since $q\le M/n<1-\sigma$,
we have $t>\sigma$ and $|T|<(t-\sigma)n$. In any supported
colouring of $W$, points outside $T$ must receive labels in $J$.
As these labels are light and $a/4\le b/n$,
\begin{align*}
|W|\le |W\cap T|+\sum_{i\in J}|\varphi^{-1}(i)|
 \le |W\cap T|+q|W|.\tag*{\qedhere}    
\end{align*}
\end{proof}

For each fingerprint, the container and its pair $(T,t)$ can
be fixed before sampling. The next estimate excludes the required
density excess simultaneously over all small fingerprints.
Conditioning on the integers of a fingerprint changes the mean
by at most its size and leaves the other integers independent.

\begin{lemma}\label{lem:palette-probability}
Fix an integer $R\ge1$ and $0<\sigma<1$, and choose $\lambda>0$
such that
\[
 \lambda\le\sigma/2,\qquad
 \lambda\log\frac{eR}{\lambda}\le\frac{\sigma^2}{48}.
\]
Let $p=p(n)\in(0,1]$ satisfy $np\to\infty$. Let $\mathcal S$
be a deterministic family of subsets of $[n]\times[R]$, each of
size at most $\lambda np$ and containing at most one labelled
copy of each integer. For each $S\in\mathcal S$, fix a
deterministic pair $T_S\subseteq[n]$, $t_S\in[0,1]$ such that
$|T_S|\le(t_S-\sigma)n$, and let $\pi(S)$ denote the projection
of $S$ onto $[n]$. For all sufficiently large $n$,
\[
 \Pr\!\left(\exists S\in\mathcal S:
 \pi(S)\subseteq\Omega,\quad
 |\Omega\cap T_S|\ge t_S|\Omega|\right)
 \le e^{-\sigma^2np/24},\qquad \Omega=[n]_p.
\]
\end{lemma}

\begin{proof}
Fix $S\in\mathcal S$, write $F=\pi(S)$ and $s=|S|=|F|$, and set
\[
 Z=\sum_{x=1}^n
       (\mathbf1_{T_S}(x)-t_S)\mathbf1_\Omega(x).
\]
Its mean is $p(|T_S|-t_Sn)\le-\sigma np$.
Conditional on $F\subseteq\Omega$, the mean increases by at
most $s\le\lambda np\le\sigma np/2$. The remaining summands
are independent, their centered absolute values are at most
one, and their total variance is at most $np$. Bernstein's
inequality therefore gives
\[
 \Pr(Z\ge0\mid F\subseteq\Omega)
 \le e^{-\sigma^2np/12}.
\]
For each $F$ of size $s$, there are at most $R^s$ choices for
$S$, while $\Pr(F\subseteq\Omega)=p^s$. A union bound gives
\[
 \Pr\bigl(\exists S\in\mathcal S:
    \pi(S)\subseteq\Omega,\ |\Omega\cap T_S|\ge t_S|\Omega|\bigr)
 \le e^{-\sigma^2np/12}
    \sum_{s=0}^{\lfloor\lambda np\rfloor}\binom ns(Rp)^s.
\]
Since $\lambda\le1/2\le R$, the function
$s\log(eRnp/s)$ is increasing for $0<s\le\lambda np$.
Consequently, the sum is at most
\[
 (\lambda np+1)
 \exp\left(\lambda np\log\frac{eR}{\lambda}\right).
\]
The hypothesis bounds the exponent by $\sigma^2np/48$.
Since $np\to\infty$, the remaining factor is at most
$\exp(\sigma^2np/48)$ for all large $n$. This proves the bound.
\end{proof}

We now combine these ingredients. The degree bounds allow
fingerprints much smaller than $np$, and the dense property
keeps each container below the edge threshold needed for its
certificate. The probability estimate then applies to all
certificates at once.

\begin{proof}[Proof of Proposition~\ref{prop:bounded-palette}]
Put
$G_n=G_n(A,\cI;H,R)$, $N=Rn$ and $E=e(G_n)$.
Lemma~\ref{lem:degrees} gives a constant $K_1>0$ such that
\begin{equation}\label{eq:labelled-degrees}
 \beta n^d\le E\le R^kn^d,\qquad
 \Delta_j(G_n)\le K_1n^{d-1-(j-1)/m(A)}
 \quad(1\le j\le k).
\end{equation}
For the lower bound on $E$, colour all of $[n]$ with the
protected label $1$ and apply the dense property fixed at
the start of this subsection.
The upper bounds use at most $R^k$ label lists per ordered
solution; no edge contains two labelled copies of one integer.

Let
\[
 \mathcal F=\{U\subseteq V(G_n):e(G_n[U])\ge\beta n^d/4\}.
\]
This family is increasing. By~\eqref{eq:labelled-degrees}, every
$U\in\mathcal F$ satisfies
\[
 |U|\ge\frac{\beta n}{4K_1},\qquad
 e(G_n[U])\ge\frac{\beta E}{4R^k}.
\]
Choose $0<\varepsilon_0<1$ smaller than both
$\beta/(4K_1R)$ and $\beta/(4R^k)$, and take
$K\ge\max\{1,RK_1/\beta\}$. Then
$|U|>\varepsilon_0N$ and $e(G_n[U])>\varepsilon_0E$ for every
$U\in\mathcal F$, while for $\tau\ge n^{-1/m(A)}$,
\[
 \Delta_j(G_n)\le K\tau^{j-1}\frac EN
 \qquad(1\le j\le k).
\]
Let $B$ be the constant from Theorem~\ref{thm:containers} for
these fixed parameters. Choose $0<\alpha\le1/2$ so small that,
with $\lambda=B\alpha R$,
\[
 \lambda\le\sigma/2,\qquad
 \lambda\log\frac{eR}{\lambda}\le\frac{\sigma^2}{48},\qquad
 \lambda K_1\le\beta/4.
\]
Take $C\ge1/\alpha$. For $Cn^{-1/m(A)}\le p\le1$, set
$\tau=\alpha p$. Then $n^{-1/m(A)}\le\tau<1$, so the container
theorem applies. For each fixed $n,p$, choose its fingerprints
$\mathcal S$ and map $f$ before sampling $\Omega=[n]_p$.

Every fingerprint has size at most
$B\tau N=\lambda np$. Since $f(S)\notin\mathcal F$, the container
$\mathcal C(S)=f(S)\cup S$ satisfies
\begin{equation}\label{eq:container-few}
 \begin{aligned}
 e(G_n[\mathcal C(S)])
 &\le e(G_n[f(S)])+|S|\Delta_1(G_n)<\bigl(\beta/4+\lambda K_1p\bigr)n^d
 \le\beta n^d/2.
 \end{aligned}
\end{equation}
Call $S$ \emph{compatible} if it contains at most one labelled
copy of each integer. For every compatible fingerprint, fix
$T_S,t_S$ from Lemma~\ref{lem:container-certificate} for
$\mathcal C(S)$. These pairs are deterministic and depend on
the fingerprints, not on the random sample.

Suppose that a colouring $\varphi:\Omega\to[R]$ has every
light class of size at most $a|\Omega|/4$ and no admissible
solution. Then $X_\varphi$ is independent in $G_n$, so
$S\subseteq X_\varphi\subseteq\mathcal C(S)$ for some compatible
fingerprint $S$, with $\pi(S)\subseteq\Omega$. Its certificate
forces $|\Omega\cap T_S|\ge t_S|\Omega|$.
Lemma~\ref{lem:palette-probability} therefore bounds the
probability of any such colouring by $e^{-\sigma^2np/24}$.
Its estimates hold uniformly in the permitted range of $p$,
since $m(A)>1$ gives $np\ge Cn^{1-1/m(A)}\to\infty$.
\end{proof}

The fixed-palette conclusion now finishes the argument for
arbitrary colourings.

\begin{proof}[Proof of Theorem~\ref{thm:one}]
Choose $\sigma,a,\beta$ from Lemma~\ref{lem:densecount}.
Use this value of $a$ in Lemma~\ref{lem:reduction} to fix
$H<R$, and apply
Proposition~\ref{prop:bounded-palette} with these fixed parameters.
With probability at least $1-e^{-cnp}$, every colouring of
$\Omega=[n]_p$ with palette $[R]$ and light classes of size at
most $a|\Omega|/4$ has an admissible solution.
On this event, let $\chi$ be any colouring of $\Omega$.
Lemma~\ref{lem:reduction} gives a colouring
$\varphi:\Omega\to[R]$ with the required light-class bound.
The proposition supplies a proper admissible solution for
$\varphi$, and the reduction ensures that it is canonical
for $\chi$. 
\end{proof}

\section*{Acknowledgement}
We thank Mathias Schacht for introducing this question to us, and we also thank Hong Liu
for helpful discussions.
ZL was supported by the Natural Science Foundation of Shanghai (No. 25ZR1402390), and GW was supported by the National Natural Science Foundation of China (No.~12671419), Natural Science Foundation of Shandong Province (No.~ZR2024MA073). ZY was supported by the Institute for Basic Science (IBS-R029-C4).

\section*{AI usage}

During the preparation of this paper, the authors used GPT 6.0 Pro to
discuss technical steps in the proofs and to revise the exposition.
The proof of the two-block $0$-statement in
Sections~\ref{sec:sparse-rank} and \ref{sec:two-zero}
was initially generated by GPT 6.0 Pro. The authors' contribution to this part was
to understand and digest the argument, reorganise the structure and rewrite the whole proof. The authors also distil the core idea of the proof (as the extension theorem) and find the connection to the Freiman problem. 
For the 1-statement, the authors independently find the colour reduction followed by
containers, inspired by the work of Allen, B\"ottcher, Lada and
Mergoni Cecchelli~\cite{ABLM}. GPT 6.0 Pro was used to finish several technical lemmas and simplify the proofs. The Schur part (Section~\ref{sec:schur}) was proved by the authors.
The authors take full responsibility
for the arguments and text.

\appendix
\section{Uniform canonical supersaturation}\label{app:finite}

We prove Lemma~\ref{lem:canonical-supersaturation} using Lefmann's
theorem for Deuber sets~\cite{Lefmann}, in the form recorded by
Pr\"omel~\cite[p.~14]{Promel}. The construction retains all
$d=k-\ell$ degrees of freedom of a solution, allowing us to count
canonical solutions without bounding the palette.

A positive $(m,P,c_0)$-set has generators $u_0,\ldots,u_m\in\N$
and rows
\[
 D_r=\left\{c_0u_r+\sum_{t>r}\lambda_tu_t:
       \lambda_t\in\Z,\ |\lambda_t|\le P\right\}
 \subseteq\N
 \qquad(0\le r\le m).
\]

\begin{theorem}[Lefmann]\label{thm:lefmann}
For $m,P,c_0\in\N$, there are $m',P',c'\in\N$ such that
every colouring of every positive $(m',P',c')$-set contains an
$(m,P,c_0)$-set $D=\bigcup_{r=0}^m D_r$ on which the
colouring is either monochromatic, injective, or constant on each
row $D_r$ with different colours on different rows.
\end{theorem}

\begin{proof}[Proof of Lemma~\ref{lem:canonical-supersaturation}]
Write $A=(\mathbf a_1,\ldots,\mathbf a_k)$,
$\cI=(I_0,\ldots,I_s)$ and $d=k-\ell$. We first choose solution
forms that fit the rows of a Deuber set. For $1\le t\le s$,
choose rational $b_{t,i}$ satisfying
\[
 \sum_{j\in I_t}\mathbf a_j
 =\sum_{i\in I_0\cup\cdots\cup I_{t-1}}b_{t,i}\mathbf a_i.
\]
Choose $c_0\in\N$ clearing their denominators. Extend a proper
positive integer kernel vector to a $\Q$-basis
$v^{(1)},\ldots,v^{(d)}$ of $\ker A$ consisting of integer
vectors, with $v^{(d)}$ proper and positive.

Put $q=s+d$ and, for $i\in I_r$, define
\[
 L_i(\mathbf u)=c_0u_r-c_0\sum_{t=r+1}^{s}b_{t,i}u_t
                    +\sum_{j=1}^{d}v_i^{(j)}u_{s+j}.
\]
Let $C$ be the coefficient matrix, so
$(L_1,\ldots,L_k)^{\mathsf T}=C\mathbf u$.
The columns condition gives $AC=0$, and the last $d$ columns
form a basis of $\ker A$. Hence $\rank C=d$.

Let $B_0\ge1$ bound the absolute values of all entries of $C$,
and choose an integer $P_0>2c_0B_0$. In a positive
$(q,P_0,c_0)$-set, $L_i(\mathbf u)$ belongs to row $r$ for
$i\in I_r$. The resulting solution is proper: $L_i-L_j$ is a
nonzero integer form because its last coefficient is
$v_i^{(d)}-v_j^{(d)}\ne0$. If its first nonzero coefficient has
index $h<q$, positivity of the target rows gives
\[
 u_h>\frac{P_0}{c_0}\sum_{t>h}u_t
       >2B_0\sum_{t>h}u_t,
\]
so the leading term cannot cancel; the case $h=q$ is immediate.
The three outcomes of Theorem~\ref{thm:lefmann} therefore give
a proper canonical solution, with the row-constant outcome
giving precisely the prescribed partition $\cI$.

Apply that theorem with target parameters $(q,P_0,c_0)$ and
write $(q',P',c')$ for the source parameters. Put $D=q'+1$.
Let $\mathcal W\subseteq\Z^D$ be the finite set of coefficient
vectors of the source forms, and set $K=|\mathcal W|$.
Every $w\in\mathcal W$ has first nonzero coefficient $c'$;
its position specifies the source row.

To obtain many source configurations in $[n]$, vary their
generators in a fixed product box $\mathcal B\subset(0,1)^D$
of positive volume. Choose it so that every source form takes
values in $(0,1)$ and
\[
 c't_j>3P'\sum_{r>j}t_r
 \qquad(0\le j<q',\ \mathbf t\in\mathcal B).
\]
Such a box exists: choose sufficiently separated positive
coordinates, scale them, and take a small product neighbourhood.
For large $n$, the set
$\mathcal T_n=n\mathcal B\cap\Z^D$ has at least $\kappa n^D$
elements for some $\kappa>0$. Every associated source set lies
in $[n]$. Moreover, the largest value in each source row is
less than twice its smallest value.

For a fixed missing point $x\in[n]\setminus S$ and a fixed
source form, at most $n^{D-1}$ parameter tuples give that value:
one coordinate is determined by the others. Choose
$0<\sigma<\min\{1/2,\kappa/(2K)\}$. If
$|S|\ge(1-\sigma)n$, at least $\kappa n^D/2$ tuples in
$\mathcal T_n$ give source sets contained in $S$.
Any colouring of $S$ restricts to each such set. For each tuple,
choose a target set supplied by Theorem~\ref{thm:lefmann} and
the proper canonical solution given by the forms $L_i$.

It remains to bound how often one solution is obtained.
Each target centre has a representation
$c_0u_r=w_r\cdot\mathbf t$ with $w_r\in\mathcal W$.
For $r<t$, target positivity gives $u_r>2u_t$, since
$P_0>2c_0$. Thus two target centres cannot come from the
same source row. Their coefficient vectors $w_0,\ldots,w_q$
have distinct leading positions and are linearly independent.
Let $W$ have these vectors as rows, so it has full row rank. Then
\[
 \mathbf u=\frac1{c_0}W\mathbf t,\qquad
 \mathbf x=\frac1{c_0}CW\mathbf t,\qquad
 \rank(CW)=d.
\]
There are at most $K^{q+1}$ possible choices of $W$.
For each such choice arising above and each fixed ordered
solution $\mathbf x$, the parameter fibre, if nonempty, has
dimension $D-d$. Fixing $D-d$ free coordinate positions
determines the remaining coordinates, so there are at most
$n^{D-d}$ parameter tuples in $[n]^D$. Each solution is
therefore counted at most $K^{q+1}n^{D-d}$ times.
The $\kappa n^D/2$ tuples consequently yield at least $\beta_0n^d$ solutions where $ \beta_0=\frac{\kappa}{2K^{q+1}}>0$.
\end{proof}

\section{Proofs of the three-variable lemmas}\label{app:three-variable-proofs}

We use the notation and component types of
Subsection~\ref{subsec:zero-ternary}.
For a basis $\mathcal B$ chosen from the solution relations on a
component $C$,
let $\Gamma(C,\mathcal B)$ be the incidence graph on $C$ of
$(\operatorname{supp}R)_{R\in\mathcal B}$ (Section~\ref{sec:sparse-rank}).
Each relation vertex has degree three; $\mathcal H_q(U)$ records
all solution supports instead.

\begin{proof}[Proof of Lemma~\ref{lem:zero-three-structure}]
\claimsof{lem:zero-three-structure}
Fix a component $C$ containing at least one hyperedge, and write
$\mathcal R_C$ for the solution relations supported in $C$.

\begin{claim}\label{cl:zero-connected-basis}
There is a basis $\mathcal B=\{R_1,\ldots,R_e\}\subseteq\mathcal R_C$
whose supports cover $C$ and whose incidence graph
$\Gamma(C,\mathcal B)$ is a tree or a unicycle.
\end{claim}

\begin{poc}
Since $B=(-1,-q,1)$ has rank one,
Lemma~\ref{lem:zero-retained-exploration} supplies a basis
$\mathcal B\subseteq\mathcal R_C$ whose supports cover $C$ and whose
incidence graph is connected.

Set $v=|C|$ and $e=|\mathcal B|$. Since $\mathcal B$ spans the
relations on $C$, equation~\eqref{eq:component-freedom} gives
$\dim(V_B(U)|_C)=v-e$.
The assumed rank bound gives $v\le2(v-e)$, hence $v\ge2e$.
The incidence graph has $v+e$ vertices and $3e$ edges.
Its cycle rank is therefore
 $3e-(v+e-1)=2e+1-v\le1.$
Thus it is a tree or a unicycle.
\end{poc}

To account for relations outside the basis, we need two restrictions
on small coordinate sets.

\begin{claim}\label{cl:zero-four-points}
Three solution relations on distinct triples of a four-point set
are linearly independent.
\end{claim}

\begin{poc}
Every solution relation has coefficient sum $-q$ and squared
Euclidean norm $q^2+2$. Distinct relations are not proportional,
since their sums would force the factor to be one.

Suppose that three distinct relations $S_1,S_2,S_3$ are dependent.
Since $S_1,S_2$ are independent, write $S_3=\alpha S_1+\beta S_2$.
Comparing coefficient sums gives $\alpha+\beta=1$, so
$S_3=(1-\beta)S_1+\beta S_2$. Expanding squares gives
\[
 \|S_3\|^2
 =(1-\beta)\|S_1\|^2+\beta\|S_2\|^2
   -\beta(1-\beta)\|S_1-S_2\|^2.
\]
The equal norms and $S_1\ne S_2$ give $\beta(1-\beta)=0$.
Thus $S_3=S_1$ or $S_3=S_2$, a contradiction.
\end{poc}

Positivity prevents cancellation at both common coordinates of two
relations.

\begin{claim}\label{cl:zero-common-pair}
Solution relations on $\{u,v,a\}$ and $\{u,v,b\}$, with four
distinct positive values, cannot have proportional coefficient
pairs on the common coordinates $\{u,v\}$.
\end{claim}

\begin{poc}
Let the proportionality factor be $\lambda\ne0$.
Each relation has coefficients given by a permutation of
$(1,-1,-q)$. Suppose first that $\lambda>0$. If a common
coordinate has a positive coefficient, that coefficient is $1$
in both relations, forcing $\lambda=1$. If neither common
coefficient is positive, their product is $q$ in each relation.
Thus $\lambda^2q=q$, again giving $\lambda=1$.
The common pairs, and hence the private coefficients, agree.
Subtracting the two evaluated equations gives a non-zero multiple
of $a-b$, so $a=b$, a contradiction.

Suppose instead that $\lambda<0$. The common coefficients have
opposite signs in the two relations. Each relation has only one
positive coefficient, so, after ordering the common coordinates,
the two coefficient pairs are $(1,-a_0)$ and $(-b_0,1)$, where
$a_0,b_0\in\{1,q\}$.
Proportionality gives $a_0b_0=1$. Since $q>0$, this forces
$a_0=b_0=1$, and hence $\lambda=-1$. Both private coefficients
are then $-q$. Adding the evaluated equations gives
$-q(a+b)=0$, impossible for positive $a,b$.
\end{poc}

For a subfamily of the basis, keep precisely the numbers in its
supports. We count those at which cancellation is impossible.

\begin{claim}\label{cl:zero-leaf-count}
Take a nonempty subfamily of $\mathcal B$. A connected component
of its incidence graph containing $f$ relations has at least
$f+2$ number vertices of degree one if it is a tree, and at least
$f$ if it is a unicycle.
\end{claim}

\begin{poc}
Let $v'$ be the number of point vertices and $L$ count those of degree one.
There are $3f$ edges, since each relation vertex has degree three.

If the component is a tree, then
$3f=f+v'-1$, so $v'=2f+1$.
If it is a unicycle, then $3f=f+v'$, so $v'=2f$.
Every other number vertex has degree at least two, so
 $3f\ge L+2(v'-L)$, and $L\ge2v'-3f$.
Substitute the two values of $v'$.
\end{poc}

A degree-one number cannot cancel in a combination with non-zero
coefficients. This identifies the relations outside the basis.

\begin{claim}\label{cl:zero-pasch-completion}
Every solution relation outside $\mathcal B$ is supported on the
three private vertices of a loose triangle of basis supports.
\end{claim}

\begin{poc}
For $S\in\mathcal R_C\setminus\mathcal B$, let $\Gamma'$ be the
incidence graph of the $f$ basis relations with non-zero coefficients
in its expansion. Each degree-one number vertex survives in
$\operatorname{supp}S$, so there are at most three.

First, $\Gamma'$ is connected. Otherwise, since it has at most
one cycle in total, it contains either two tree components or a tree
component and a unicycle component. By
Claim~\ref{cl:zero-leaf-count}, two tree components give at least
six leaves. A tree component gives at least three leaves and a
unicycle component, which uses at least two relations, gives at
least two. Both possibilities leave more than three coordinates
in $S$.

If $\Gamma'$ is a tree, Claim~\ref{cl:zero-leaf-count} gives
$f+2\le3$, so $f=1$. Then $S$ is a scalar multiple of one basis
relation. Their coefficient sums are both $-q$, forcing the scalar
to be one. This contradicts $S\notin\mathcal B$.

Thus $\Gamma'$ is a unicycle, and Claim~\ref{cl:zero-leaf-count}
gives $2\le f\le3$.

If $f=2$, the two basis triples share two coordinates and have
four vertices in total. Their two private coordinates both survive
in $S$. Its support is therefore different from either basis
support. The two basis relations and $S$ are three dependent
relations on distinct triples of four points, contrary to
Claim~\ref{cl:zero-four-points}.

Let $f=3$. If the unique cycle visits all three relation vertices,
their supports form a loose triangle: each pair meets in a different
cycle vertex, and each relation has one further, private vertex.
These three private coordinates survive, so they form precisely
$\operatorname{supp}S$.

The other possibility is that the cycle uses two relation vertices.
Their triples share two coordinates, say
$\{u,v,a\}$ and $\{u,v,b\}$. The third triple attaches at one
coordinate and introduces two new coordinates. If it attaches
at $u$ or $v$, there are four number vertices of degree one,
which is impossible. After interchanging the first two triples
if necessary, the only remaining form is
\[
 \{u,v,a\},\qquad \{u,v,b\},\qquad \{a,c,d\}.
\]
The three leaves are $b,c,d$, so these must be the support of $S$.
The coefficients at $u,v$ must both cancel. Only the first two
relations involve these coordinates, so their restrictions to
$\{u,v\}$ would be proportional, contrary to
Claim~\ref{cl:zero-common-pair}.
Thus only the loose triangle is possible.
\end{poc}

\begin{claim}
There is at most one relation in $\mathcal R_C\setminus\mathcal B$.
Moreover, each triple supports at most one distinct solution relation.
\end{claim}

\begin{poc}
The graph $\Gamma(C,\mathcal B)$ has at most one cycle, so it
contains at most one loose triangle of basis supports.
No two basis relations have the same support: their incidence
subgraph would be $K_{2,3}$, of cycle rank two.
By Claim~\ref{cl:zero-pasch-completion}, every additional relation
must use precisely the private vertices of that triangle.

Write the three basis relations on the triangle as $R_1,R_2,R_3$.
To cancel a common coordinate, the coefficients of its two incident
relations must have a fixed ratio. Two such cancellations already
determine the three combination coefficients up to a common factor.
The third cancellation is either consistent with these ratios or
impossible. Finally, membership in $\mathcal R_q(U)$ fixes the
coefficient sum to be $-q$, so
$\lambda_1+\lambda_2+\lambda_3=1$.
There is therefore at most one additional relation.

Its support cannot already be a basis support. Indeed, the three
triangle supports together with their private-vertex triple form
a Pasch configuration. Their incidence graph has ten vertices and
twelve edges, hence three edges outside a spanning tree. It cannot
be a subgraph of $\Gamma(C,\mathcal B)$, which has at most one.
Thus an additional relation cannot have the same support as a
basis relation. Together with the preceding uniqueness, this shows
that every support carries at most one distinct solution relation.
\end{poc}

If there is no additional relation, the solution hypergraph on $C$
has the tree or unicycle incidence graph supplied by
Claim~\ref{cl:zero-connected-basis}.
If there is an additional relation, it completes the unique loose
triangle to a Pasch configuration, as in
Figure~\ref{fig:zero-pasch}. All other basis triples belong to trees
attached to this configuration, because they lie outside the unique
cycle of $\Gamma(C,\mathcal B)$.
Isolated vertices are one-vertex incidence trees, completing the proof.
\end{proof}

\begin{proof}[Proof of Lemma~\ref{lem:zero-three-colourings}]
\claimsof{lem:zero-three-colourings}
We use the same palette $\{1,2,3\}$ in every component.

\begin{claim}\label{cl:zero-tree-extension}
A prescribed colour at the attachment vertex of an incidence tree
extends to either kind of colouring of that tree.
\end{claim}

\begin{poc}
Root the incidence graph at the attachment vertex. Add its triples
outwards from the root. Each new triple contains one vertex already
coloured and two new vertices.

For a strong colouring, give the new vertices the two remaining
colours. For the second colouring, suppose first that the old
vertex is distinguished. Give one new vertex its colour and the
other a different colour. If the old vertex is not distinguished,
give the other non-distinguished vertex a different colour, and
give the distinguished vertex either of these two colours.
Each step preserves all earlier colours.
\end{poc}

\begin{claim}\label{cl:zero-strong-colouring}
Every component of $\mathcal H$ has a strong three-colouring.
\end{claim}

\begin{poc}
An incidence tree is fixed by
Claim~\ref{cl:zero-tree-extension}, after assigning colour to its root.

For an incidence unicycle, list the triples meeting its unique
cycle in cyclic order. They are
\[
 T_i=\{u_i,u_{i+1},w_i\}\quad(1\le i\le t),
 \qquad u_{t+1}=u_1,\qquad t\ge2.
\]
The vertices $u_i$ lie on the cycle. The vertices $w_i$ are
distinct and outside it; any remaining triples form attached trees.
Give the $u_i$ a proper colouring of the cycle with at most three
colours. For $t=2$, this simply means that $u_1,u_2$ have different
colours. Give $w_i$ the colour missing from $u_i,u_{i+1}$.
Every $T_i$ is now rainbow.

For a Pasch configuration, label its triples $T_1,\ldots,T_4$
and write $x_{ij}$ for the common vertex of $T_i$ and $T_j$.
Assign the same colour to each of the following pairs:
\[
 \{x_{12},x_{34}\},\qquad
 \{x_{13},x_{24}\},\qquad
 \{x_{14},x_{23}\},
\]
using a different colour for each pair. Each $T_i$ meets each
pair once, so every triple is rainbow.
In both cases, Claim~\ref{cl:zero-tree-extension} extends the
colouring over every attached tree.
\end{poc}

\begin{claim}\label{cl:zero-marked-trees}
Incidence trees and unicycles have the second kind of three-colouring.
\end{claim}

\begin{poc}
For trees use Claim~\ref{cl:zero-tree-extension}.
On a unicycle, write the cycle triples as
$T_i=\{u_i,u_{i+1},w_i\}$ and properly colour the $u_i$ as in
Claim~\ref{cl:zero-strong-colouring}. Thus the two cycle vertices
of every $T_i$ have different colours.
If $w_i$ is distinguished, give it either endpoint colour.
Otherwise one of $u_i,u_{i+1}$ is distinguished; give $w_i$ the
colour of that distinguished endpoint. In each case the other
two vertices have different colours, and the distinguished vertex
repeats one of them. Extend over all attached trees using
Claim~\ref{cl:zero-tree-extension}.
\end{poc}

\begin{claim}\label{cl:zero-marked-pasch}
Every choice of distinguished vertices in a Pasch configuration
admits the second kind of three-colouring.
\end{claim}

\begin{poc}
On the six vertices $x_{ij}$, form $D$ by joining the
non-distinguished pair in each $T_i$. Its four edges are distinct,
since a pair of Pasch vertices belongs to at most one triple.

If $D$ is bipartite, properly colour it with two colours.
The two non-distinguished vertices of each triple then have
different colours. Its distinguished vertex, having one of the
same two colours, necessarily repeats one of them.

Suppose that $D$ is not bipartite. With only four edges, an odd
cycle in $D$ must be a triangle. Regard the labels $ij$ as edges
of $K_4$. Its three triangle vertices are pairwise incident
edges of $K_4$, so they form either a star or a triangle in $K_4$.
They cannot form a star: all three pairs would then belong to
the same $T_i$, whereas $T_i$ contributes only one edge to $D$.

Relabel so that the triangle vertices are $x_{12},x_{23},x_{13}$.
The distinguished vertices of $T_1,T_2,T_3$ are then
$x_{14},x_{24},x_{34}$, respectively.
By permuting $1,2,3$, we may assume that the distinguished vertex
of $T_4$ is $x_{14}$. Set
\[
 \begin{gathered}
 \chi(x_{12})=1,\qquad \chi(x_{23})=2,\qquad \chi(x_{13})=3,\qquad
 \chi(x_{14})=\chi(x_{24})=1,\qquad \chi(x_{34})=2.
 \end{gathered}
\]
The four requirements are:
\begin{center}
\footnotesize
\begin{tabular}{@{}ccc@{}}
\toprule
Triple & Colours of the other two vertices & Distinguished colour\\
\midrule
$T_1$ & $(1,3)$ & $1$\\
$T_2$ & $(1,2)$ & $1$\\
$T_3$ & $(3,2)$ & $2$\\
$T_4$ & $(1,2)$ & $1$\\
\bottomrule
\end{tabular}
\end{center}
Each row satisfies the required condition.
\end{poc}

Claim~\ref{cl:zero-marked-pasch} colours each Pasch configuration,
and Claim~\ref{cl:zero-tree-extension} extends over its attached
trees. Together with Claims~\ref{cl:zero-strong-colouring}
and~\ref{cl:zero-marked-trees}, this proves both assertions.
\end{proof}

\section{Freiman dimension in sparse random sets}\label{app:freiman}

For a finite subset $A$ of an abelian group, let $V(A)$ be the
real vector space of maps $f:A\to\R$ satisfying
$f(a)+f(b)=f(c)+f(d)$ whenever $a+b=c+d$ with $a,b,c,d\in A$.
These are the Freiman homomorphisms on $A$.
Set $d_{\rm Fr}(A)=\dim V(A)-1$ for $A\ne\varnothing$ and
$d_{\rm Fr}(\varnothing)=0$.

Conlon and Gowers~\cite[Theorem~1.2]{CGFreiman} proved that,
for a finite abelian group $G$ of order $N$ and
$p=C(\log N)^{1/3}N^{-2/3}$ with $C$ sufficiently large,
with high probability every Freiman homomorphism from $G_p$
to an abelian group extends to $G$. Real-valued maps are then
constant, so $d_{\rm Fr}(G_p)=0$. They asked in Section~8 how
the dimension falls between the $N^{-2/3}$ and logarithmic
scales. Theorem~\ref{thm:sparse-rank} describes the small-constant
end, where proper additive relations are independent up to
an error bounded in probability.

Let $V_4(A)$ impose only the relations on four distinct points,
and let $Q_4(A)$ count these up to exchanging summands or sides.
Let $T(A)$ count non-trivial three-term progressions by their
centres and unordered endpoints. Write
$V|_X=\{f|_X:f\in V\}$ and let $O_{\Pr}(1)$ denote a quantity
bounded in probability.

\begin{theorem}\label{cor:freiman-application}
Fix $0<c\le1/32$ and set $p=cN^{-2/3}$. Let $S$ be either
an $N$-element subset of $\R\setminus\{0\}$, or $\Z/N\Z$
with $N$ prime, and put $A=S_p$. Then
\[
 d_{\rm Fr}(A)=|A|-1-Q_4(A)+O_{\Pr}(1).
\]
For real sets $S$, we also have $T(A)=O_{\Pr}(1)$ and, with
probability at least $1-K_4N^{-1/3}$, simultaneously for
every $X\subseteq A$,
\[
 \dim(V_4(A)|_X)\ge\frac{2|X|}{3},\qquad
 \dim(V(A)|_X)\ge\frac{2|X|}{3}-T(A).
\]
In particular,
\[
 \frac{d_{\rm Fr}(A)}{|A|}\xrightarrow{\Pr}
 \begin{cases}
  1-c^3/12,&S=[N],\\
  1-c^3/8,&S=\Z/N\Z.
 \end{cases}
\]
All dimensions are over $\R$; the ratio is set to zero when
$A$ is empty.
\end{theorem}

\begin{proof}
The matrix $B=(1,1,-1,-1)$ is regular with $m(B)=3/2$ over
$\R$ and every field of odd characteristic: deleting any two
columns preserves rank, and $\ker B$ projects onto every proper
set of coordinates.

For real hosts, $V_B(A)=V_4(A)$, so
Theorem~\ref{thm:sparse-rank} gives the first restriction bound.
After cancellation, every remaining non-trivial Freiman relation
is a three-term progression relation. Thus $V(A)$ has codimension
at most $T(A)$ in $V_4(A)$, giving the second bound. Also,
$\E T(A)=O(N^2p^3)=O_c(1)$, so $T(A)=O_{\Pr}(1)$.

Let $M_4$ have one signed row per proper relation, and choose a
row basis on the extension event over the host field $\mathbb F$.
A nonbasis row uses $f\ge2$ basis rows with nonzero coefficients,
since distinct proper rows are not proportional. Let $X$ be the
union of their supports and put $v=|X|$. By
\eqref{eq:extension-relation-bound} and
\eqref{eq:extension-supported-relations},
\[
 f\le\dim\bigl(R_B(A)\cap\mathbb F^X\bigr)
   =|X|-\dim\bigl(V_B(A)|_X\bigr)\le\frac{|X|}{3}.
\]
Hence $v\ge3f$. Among the $4f$ incidences, at least
$2v-4f\ge2f$ points occur in just one basis row. These cannot
cancel and must lie in the four-point support of the nonbasis
row. Thus $f=2$ and $v=6$: the two basis rows meet in exactly
two points, and their four private points support the nonbasis
row. In characteristic different from two, four distinct points
admit at most one proper pairing, since two different pairings
would force two points to agree. Each pair of basis rows
therefore determines at most one nonbasis row, up to sign.
Writing $P_2(A)$ for pairs of proper relations meeting in
exactly two points, we obtain
$0\le Q_4(A)-\rank M_4\le P_2(A)$.
Such a pair has six points and two independent equations, so
$\E P_2(A)=O(N^4p^6)=O_c(1)$. For real hosts, adjoining the
progression rows changes rank by at most $T(A)$; rank--nullity
gives the dimension formula. For prime cyclic hosts, zero is
absent with probability $1-p$. Apply Theorem~\ref{thm:sparse-rank}
over $\mathbb F_N$ on the nonzero host. By the minor criterion,
the integer matrix $M_4$ has real rank at least its rank modulo
$N$. The same bounds on $P_2(A)$ and $T(A)$ therefore give
the dimension formula over $\R$.

Finally, the ordered solution counts for $a+b=c+d$ in $[N]$
and $\Z/N\Z$ are $(2N^3+N)/3$ and $N^3$, respectively.
Removing $O(N^2)$ repeated-entry tuples and dividing by eight
gives $Q_4(S)=\gamma_SN^3+O(N^2)$, with $\gamma_S=1/12$
and $1/8$. Two distinct proper relations sharing $j$ points
have $O(N^{6-j})$ assignments: reserve one private entry in
each and solve for it. Hence
\[
 \Var Q_4(A)
 =O\left(N^3p^4+\sum_{j=1}^3N^{6-j}p^{8-j}\right)
 =O_c(N^{1/3}).
\]
Chebyshev's inequality and $|A|/(Np)\to1$ in probability give
$Q_4(A)/|A|\to\gamma_Sc^3$, proving the two limits.
\end{proof}

\end{document}